\documentclass[12pt,reqno]{amsart}
\usepackage{amsaddr, latexsym, amsfonts, amsmath, amsthm, amssymb, amscd,
	color}
\usepackage{pstricks}
\usepackage{pst-optic}
\usepackage{caption}
\usepackage{listings}
\usepackage{graphics}    
\usepackage{graphicx}
\usepackage{color}
\usepackage{mathrsfs}

\usepackage{xparse}
\usepackage{enumerate}
\usepackage{epstopdf}
\usepackage{hyperref}
\usepackage[backend=bibtex]{biblatex}
\bibliography{REF}

\makeatletter
\newsavebox\myboxA
\newsavebox\myboxB
\newlength\mylenA

\newcommand*\xoverline[2][0.75]{%
	\sbox{\myboxA}{$\m@th#2$}%
	\setbox\myboxB\null% Phantom box
	\ht\myboxB=\ht\myboxA%
	\dp\myboxB=\dp\myboxA%
	\wd\myboxB=#1\wd\myboxA% Scale phantom
	\sbox\myboxB{$\m@th\overline{\copy\myboxB}$}%  Overlined phantom
	\setlength\mylenA{\the\wd\myboxA}%   calc width diff
	\addtolength\mylenA{-\the\wd\myboxB}%
	\ifdim\wd\myboxB<\wd\myboxA%
	\rlap{\hskip 0.5\mylenA\usebox\myboxB}{\usebox\myboxA}%
	\else
	\hskip -0.5\mylenA\rlap{\usebox\myboxA}{\hskip
		0.5\mylenA\usebox\myboxB}%
	\fi}
\makeatother

\makeatletter
\def\subsection{\@startsection{subsection}{2}%
	\z@{1\linespacing}{.3\linespacing}%
	{\normalfont\bfseries}}
\makeatother

\newtheorem*{thm*}{Theorem}
\newtheorem{thm}{Theorem}
\newtheorem{prop}[thm]{Proposition}

\newtheorem{cor}[thm]{Corollary}

\newtheorem{lem}[thm]{Lemma}
\newtheorem*{lem*}{Lemma}

\newtheorem*{conj*}{Conjecture}

\theoremstyle{remark}
\newtheorem{rmk}{Remark}[section]

\numberwithin{equation}{section}
\numberwithin{thm}{section}

\def\sgn{\operatorname{sgn}}

\def\det{\operatorname{det}}
\def\perm{\operatorname{perm}}

\newmuskip\pFqskip
\mathchardef\pFcomma=\mathcode`, % keep a copy of the comma

\newcommand*\pFq[5]{%
	\begingroup
	\begingroup\lccode`~=`,
	\lowercase{\endgroup\def~}{\pFcomma\mkern\pFqskip}%
	\mathcode`,=\string"8000
	{}_{#1}F_{#2}\biggl[\genfrac..{0pt}{}{#3}{#4};#5\biggr]%
	\endgroup
}
\begin{document}
	
	\title{Inverting the Kasteleyn matrix for holey hexagons}
	
	\author[T. Gilmore]{Tomack
		Gilmore$^\dagger$}
	\address{Fakult\"at f\"ur Mathematik der Universit\"at Wien,\\
		Oskar-Morgernstern-Platz 1, 1090 Wien, Austria.}
	\email{tomack.gilmore@univie.ac.at}
	\thanks{$\dagger$Research
		supported by the Austrian Science Foundation (FWF), grant F50-N15, in the
		framework of the Special Research Program ``Algorithmic and Enumerative
		Combinatorics''.}
	\begin{abstract}
		Consider a semi-regular hexagon on the triangular lattice (that is, the lattice consisting of unit equilateral triangles, drawn so that one family of lines is vertical). Rhombus (or lozenge) tilings of this region may be represented in at least two very different ways: as families of non-intersecting lattice paths; or alternatively as perfect matchings of a certain sub-graph of the hexagonal lattice. In this article we show how the lattice path representation of tilings may be utilised in order to calculate the entries of the inverse Kasteleyn matrix that arises from interpreting tilings as perfect matchings. Our main result gives precisely the inverse Kasteleyn matrix (up to a possible change in sign) for a semi-regular hexagon of side lengths $a,b,c,a,b,c$ (going clockwise from the south-west side). Not only does this theorem generalise a number of known results regarding tilings of hexagons that contain punctures, but it also provides a new formulation through which we may attack problems in statistical physics such as Ciucu's electrostatic conjecture.
	\end{abstract}
	\maketitle
	\section{Introduction}
	
	A semi-regular hexagon on the unit triangular lattice is an hexagonal region where each pair of parallel edges that comprise its outer boundary are of the same length. Such a region encloses equinumerous sets of left and right pointing unit triangles (see Figure~\ref{fig:Dual}) and by joining together all pairs of unit triangles that share exactly one edge we obtain what is known as a rhombus (or lozenge) tiling of the hexagon (Figure~\ref{fig:LineLab} shows an example of a tiling of a hexagon where two unit triangles have been removed).
	
	Rhombus tilings of hexagons have been studied in one form or another for over 100 years- in the literature perhaps the earliest result relating to these objects is MacMahon's boxed plane partition formula\footnote{Although MacMahon was originally concerned with counting plane partitions contained within a box there is a straightforward bijection that relates them to rhombus tilings of hexagonal regions.}~\cite{MacMahon16}. Since then these classical combinatorial objects have been the focus of a great deal of research and are (with respect to enumerating tilings) reasonably well understood. An excellent survey of the history of plane partitions, their symmetry classes, and their relation to rhombus tilings may be found in~\cite{KrattSTAN}. 
	
	More recently, many results have arisen concerning rhombus tilings of regions that contain gaps or holes within their interiors\footnote{These are sometimes referred to as \emph{holey hexagons}.} (see ~\cite{CiucuFischer12,CiucuFischer16,CiucuKratt13}\cite{Eisen99b,Eisen01,Eisen99a,Fischer01} to name but a few), however each individual result treats a separate and distinct class of holes. As far as the author is aware there exists to date no result that unifies these recent works, bringing them together under one roof. 
	
	Within this area perhaps the most striking result of all is a conjecture due to Ciucu~\cite{Ciucu08} (see Section~\ref{sec:App}), which draws parallels between the correlation function of holes within a ``sea of rhombi"\footnote{The correlation of holes may loosely be interpreted as a measure of the ``effect" that holes have on rhombus tilings of the plane, it is defined formally in Section~\ref{sec:App}.} and Coulomb's law for two dimensional electrostatics. This conjecture remains wide open, although it has been proved for a small number of different classes of holes (see for example~\cite{Gilmore16,Gilmore15} by the author, and~\cite{Ciucu09} for a similar result for tilings embedded on the torus). A proof of Ciucu's conjecture is thus desirable, not least because it also incorporates an analogous conjecture to that of Fisher and Stephenson~\cite{FishSteph63} (proved very recently by Dub\'edat~\cite{Dub15}).
	
	The main result of this article (Theorem~\ref{thm:Main}) arose from attempts to prove Ciucu's conjecture for a large class of holes. Roughly speaking, this result gives an exact formula for the entries of the inverse Kasteleyn matrix that corresponds to semi-regular hexagonal regions of the triangular lattice. This immediately generalises earlier formulas due to both Fischer~\cite{Fischer01} and Eisenk\"olbl~\cite{Eisen99a} that count tilings containing a fixed rhombus or pair of unit triangular holes that touch at a point. Moreover by combining Theorem~\ref{thm:Main} with an earlier result of Kenyon~\cite{Kenyon97} we also obtain an expression for the number of tilings of a region that contains holes that have even \emph{charge} (see Remark~\ref{rmk:charge} in Section~\ref{sec:RhombTil}) that involves taking the determinant of a matrix whose size is dependent on the size of the \emph{holes}, and not the size of the \emph{region to be tiled}. The class of holes for which this holds is very large, large enough, in fact, that Theorem~\ref{thm:Main} offers an alternative way to derive a large number of the enumerative results mentioned above. In the same vein, this approach may also be specialised in order to recover the generalisation of Kuo condensation described in~\cite{Ciucu15} for the regions under consideration in this article.
	
	More important than these enumerative results, however, is the potential application of Theorem~\ref{thm:Main} to a number of problems in statistical physics. After successfully extracting the asymptotics of the individual entries of the inverse Kasteleyn matrix as the size of the region tends to infinity (this has yet to be completed) we would in the first instance obtain an analogous result to that of Kenyon~\cite{Kenyon97} who considers the local statistics of fixed rhombi within tilings embedded on a torus. Further to this, under the somewhat reasonable assumption that in the limit these entries will lead to a straightforward determinant evaluation (as similar analysis showed in~\cite{Gilmore16}), Theorem~\ref{thm:Main} could very well lead to a proof of Ciucu's conjecture for the most general class of holes to date. By stretching these assumptions a little further it is not so difficult to imagine that we may also be able to obtain an alternative proof to that given by Dub\'edat of Fisher and Stephenson's conjecture from 1963.
	
	Establishing Theorem~\ref{thm:Main} relies on representing rhombus tilings of hexagons in two very different ways. In Section~\ref{sec:PerfMatch} we review a method due to Kasteleyn that allows us to count perfect matchings of a planar bipartite graph by taking the determinant of its bi-adjacency matrix (referred to as the Kasteleyn matrix of the graph). Rhombus tilings of a hexagon are in this section considered in terms of perfect matchings on a sub-graph of the hexagonal lattice, and it is here that we discuss how the number of perfect matchings of such a sub-graph that contains gaps or holes may be calculated by considering the inverse of its corresponding Kasteleyn matrix. In Section~\ref{sec:RhombTil} we move on to considering tilings as families of non-intersecting paths consisting of unit north and east steps on the (half) integer lattice. Each region that contains holes yields a corresponding lattice path matrix, and in Section~\ref{sec:Comb} we use a very recent result of Cook and Nagel~\cite{CookNagel15} to show how the lattice path matrix that arises from the path representation of tilings may be used to calculate the entries of the inverse Kasteleyn matrix corresponding to the same region. Section~\ref{sec:Exact} is then dedicated to proving an exact formula for the determinant of the corresponding lattice path matrices, from which Theorem~\ref{thm:Main} easily follows. We conclude in Section~\ref{sec:App} by discussing in a little more depth some of the potential applications of our main result.

	\section{Perfect matchings on the hexagonal lattice}\label{sec:PerfMatch}	
	We begin by discussing a method by which one may enumerate perfect matchings of bipartite combinatorial maps, originally due to Kasteleyn~\cite{Kasteleyn67}. In the following we consider planar bipartite maps, however Kasteleyn's method is in fact applicable to any planar map (see Remark~\ref{rmk:Kast} for further details).
	\subsection{Kasteleyn's method}\label{subsec:KastMeth}
	Let $G=(V,E)$ be a planar bipartite graph consisting of a set of equinumerous black and white vertices $V=\{b_1,\dots,b_n,w_1,\dots,w_n\}$ and a set of edges $E$, and suppose that $G$ is embedded on a sphere (such a graph is sometimes referred to as a \emph{planar bipartite combinatorial map}- from now on, simply a \emph{map}). A \emph{matching} of $G$ is a subset of its edges, say $E'\subseteq E$, together with the vertices with which they are incident, say $V'\subseteq V$, such that every vertex in $V'$ is incident with precisely one edge in $E'$. A matching is \emph{perfect} if $V'=V$ (see Figure~\ref{fig:bipartite}).
	
	Suppose we label the black and white vertices of $G$ from $b_1,b_2,\dots,b_n$ and $w_1,w_2,\dots,w_n$ respectively and attach to its edges taken from some commutative ring, thereby obtaining a \emph{weighted map $G_w$} where the weight of an edge that connects two adjacent vertices $b_i, w_j\in V$ is denoted by $w(b_i,w_j)$ (if $b_i,w_j$ are not adjacent then we set $w(b_i,w_j)=0$). For some  $\sigma\in\mathfrak{S}_n$ (that is, the symmetric group on $n$ letters) let $P_m(B,W_{\sigma})$ denote the product of the weights of the edges of the perfect matching in which $b_i\in B:=(b_1,b_2,\dots,b_n)$ and $w_{\sigma(j)}\in W_{\sigma}:=(w_{\sigma(1)},w_{\sigma(2)},\dots,w_{\sigma(n)})$ are adjacent. The sum over all such weighted perfect matchings of $G$ is thus
	$$\sum_{\sigma\in\mathfrak{S}_{n}}P_m(B,W_{\sigma}).$$
	
		\begin{figure}[t!]
			\includegraphics[width=0.33 \textwidth]{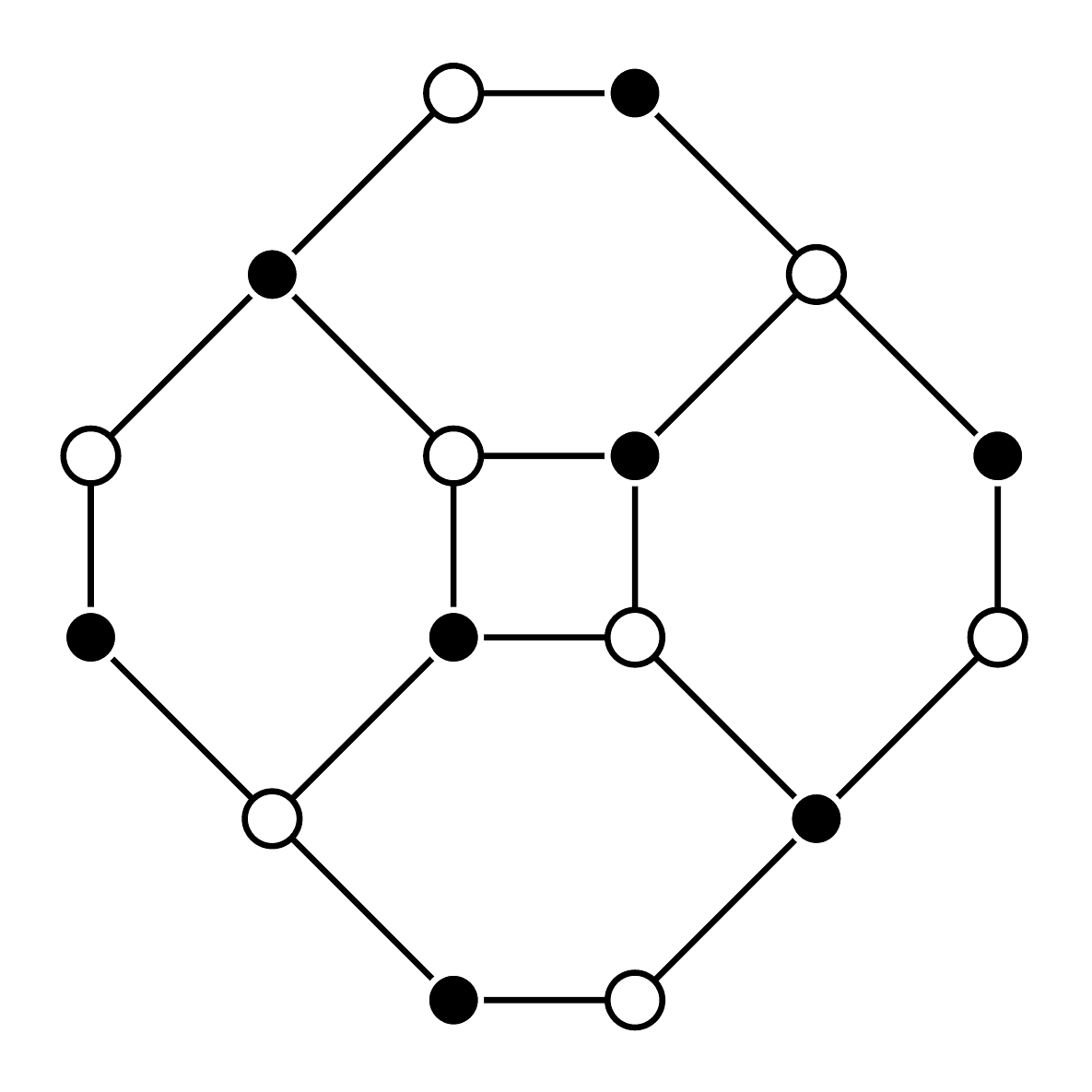}\includegraphics[width=0.33 \textwidth]{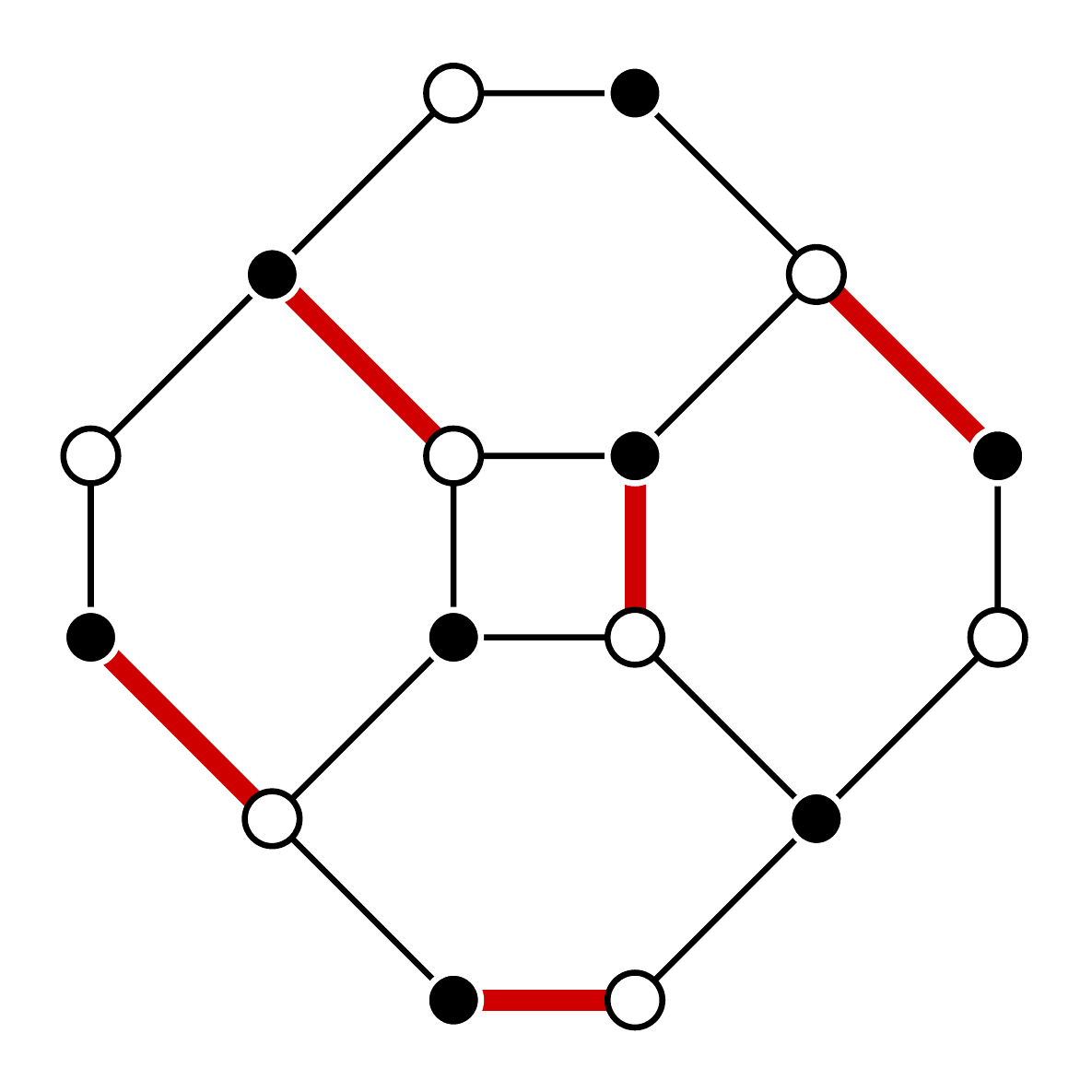}\includegraphics[width=0.33 \textwidth]{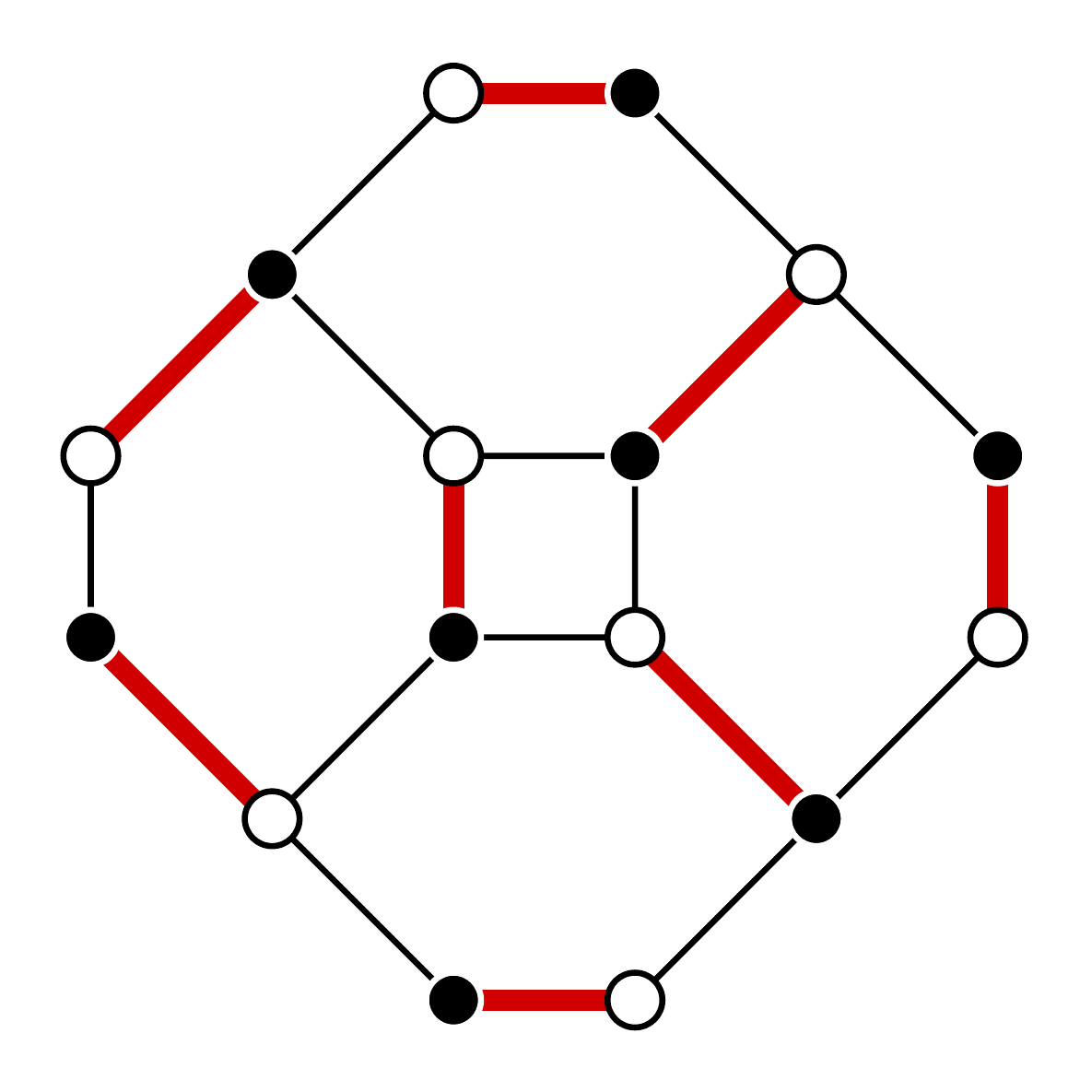}\caption{From left to right: a bipartite graph; a matching; a perfect matching.}\label{fig:bipartite}
		\end{figure}
	
Let us define the \emph{weighted bi-adjacency matrix of $G_w$} to be the $n\times n$ matrix $A_{G_w}$ with $i$-th row and $j$-th column indexed by the vertices $b_i$ and $w_j$ respectively, where each $(i,j)$- entry is given by $w(b_i,w_j)$. Then we may re-write the above expression as
\begin{equation}\label{eq:Perm}\sum_{\pi\in \mathfrak{S}_{n}}\prod_{i}^{n}(A_{G_w})_{b_i,w_{\pi(i)}},\end{equation}
which is otherwise known as the \emph{permanent} of $A_{G_w}$ (denoted $\perm(A_{G_w})$)\footnote{In order to count the number of perfect matchings we simply set all edge weights between adjacent vertices to be 1.}.

If our goal is to find a closed form evaluation for the expression in~\eqref{eq:Perm} then at first sight one may be forgiven for thinking that we have reached a dead end. The permanent of a matrix is, after all, a somewhat enigmatic function whose properties are little understood; indeed, Valiant's algebraic variations of the \emph{P vs. NP} problem~\cite{Valiant79} may be phrased in terms of the complexity of its computation. A great deal more is known, however, about the determinant of a matrix- the much loved distant relative of the permanent with a comparative abundance of useful, well-understood properties, obtained by multiplying the product in each term of the summand in~\eqref{eq:Perm} by the signature (or sign) of its corresponding permutation. 

How, then, may we relate the permanent of a matrix to its determinant? As far as the author is aware there exists no general method that allows us to express one in terms of the other, however Kasteleyn~\cite{Kasteleyn67} showed that in certain situations this is indeed possible. 

Suppose we endow the surface of the sphere on which $G_w$ is embedded with an orientation in the clockwise direction. Let us orient the edges of $G_w$ so that each edge is directed from a black vertex to a white one, thereby obtaining an \emph{oriented} weighted map (see Figure~\ref{fig:orient},left). Kasteleyn showed that for such maps it is always possible to change the direction of a finite (possibly empty) set of edges so that in each oriented face of $G_w$ an odd number of edges agree with the orientation of the surface of the sphere (when the edges are viewed from the centre of each face). Such an orientation is called \emph{admissible} and we will denote by $G_w^+$ the weighted map $G_w$ together with an admissible orientation (see Figure~\ref{fig:orient}, right). We encode such an orientation within the weighting of $G_w$ by multiplying by $-1$ the weights of those edges that are directed from white vertices to black. The weighted bi-adjacency matrix of $G_w^+$, $A_{G_w^+}$, is referred to as the \emph{Kasteleyn matrix} of $G_w$, and it follows from~\cite{Kasteleyn67} that
$$|\det(A_{G_w^+})|=|\perm(A_{G_w})|.$$
	
\begin{rmk}\label{rmk:Kast}It should be noted that Kasteleyn's method is in fact more general than it appears here. Indeed one can use a similar approach to count weighted perfect matchings of any planar graph. In this case, rather than a determinant, one considers the \emph{Pfaffian} of a weighted adjacency matrix whose rows (and also its columns) are indexed by \emph{all} the vertices of the graph. For bipartite graphs a straightforward argument shows that such a computation reduces to the situation described above.
\end{rmk}
\begin{figure}[t!]
	\includegraphics[width=0.3 \textwidth]{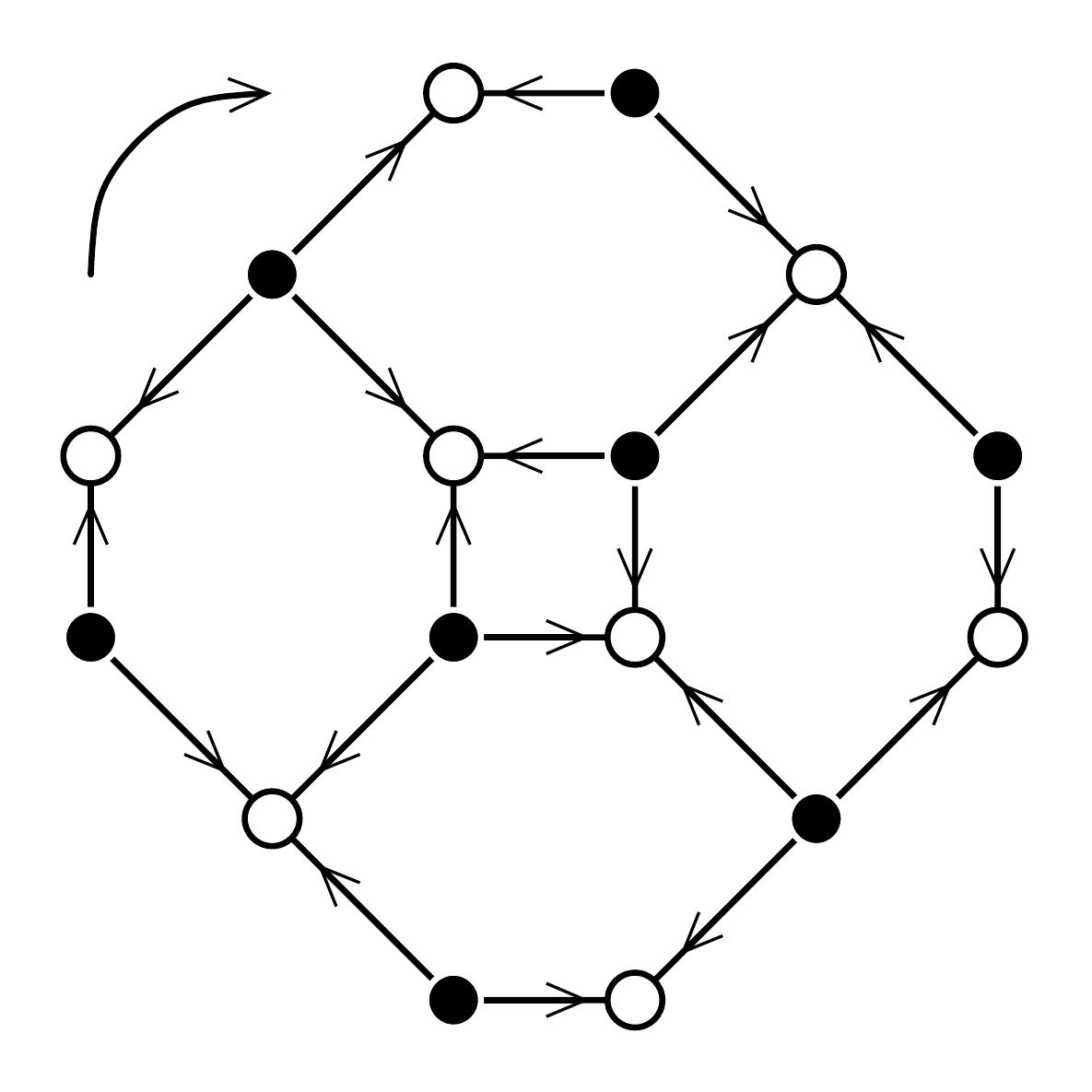}\,\,\,\,\,\,\,\,\,\includegraphics[width=0.3 \textwidth]{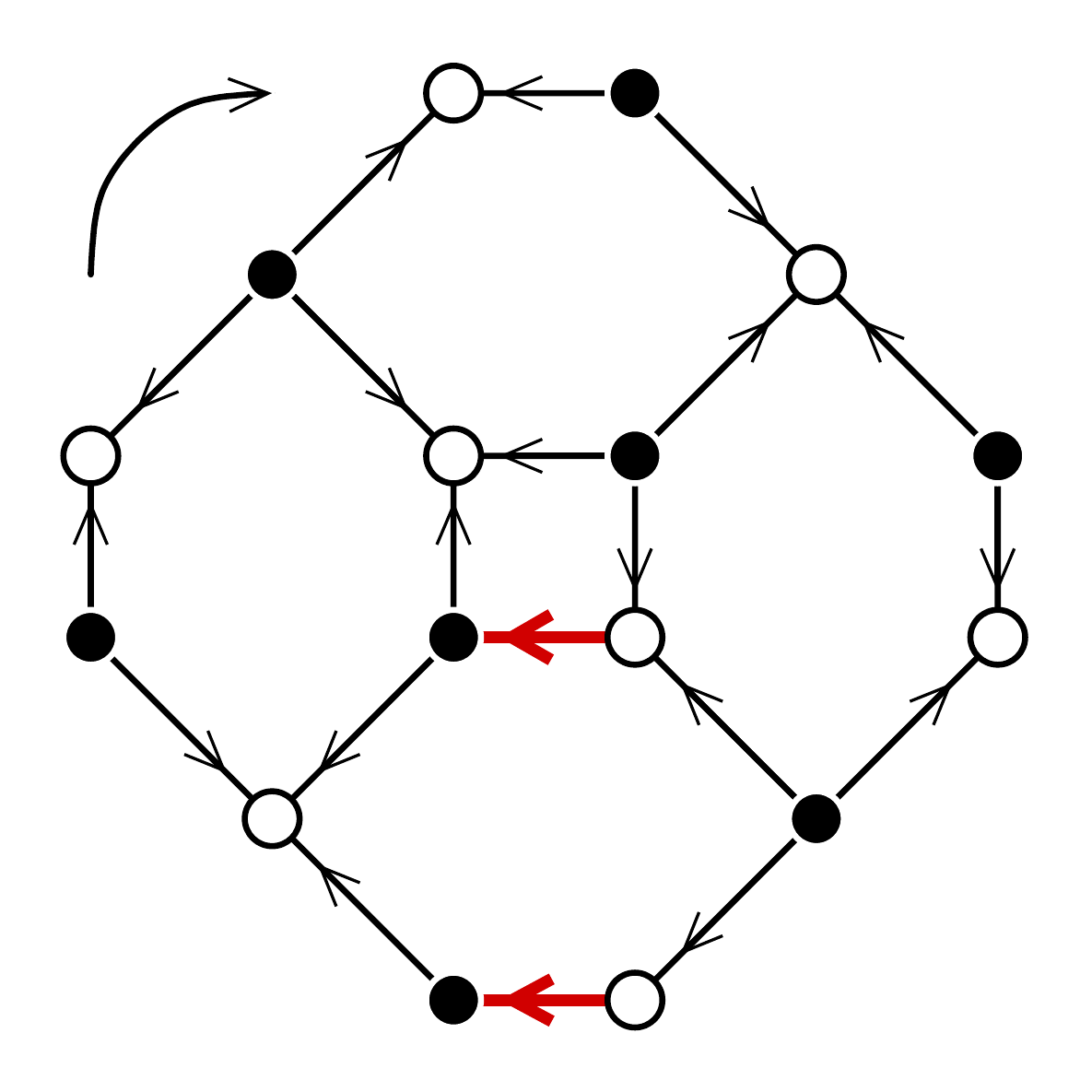}\caption{An orientation that is not admissible (left) and an admissible one obtained by changing the direction of two (red-coloured) edges (right).}\label{fig:orient}
\end{figure}

\subsection{Kasteleyn's method on the hexagonal lattice}\label{subsec:KastHex}
Imagine the plane is tiled with regular hexagons\footnote{By which we mean that all sides are of the same length.} that do not overlap nor contain any gaps, arranged so that the boundary of each hexagon contains a pair of horizontal parallel edges. Suppose we place vertices at the corners of each hexagon coloured in a chessboard fashion (that is, the vertices are coloured white and black in such a way that no vertex is adjacent to another vertex of the same colour). Let $\mathscr{H}$ denote the set of vertices and edges obtained from this tiling ($\mathscr{H}$ is often referred to as the \emph{hexagonal lattice}, see Figure~\ref{fig:HexSub}, left).

Let $G_{a,b,c}$ denote the sub-graph of $\mathscr{H}$ whose outer boundary is determined by beginning at the centre of an hexagonal face and traversing faces that share a common edge via $(a-1)$ north-west edges, then $(b-1)$ north edges, then $(c-1)$ north-east edges, then $(a-1)$ south-east edges, $(b-1)$ south edges, and finally $(c-1)$ south-west edges. Such a region shall be referred to as an \emph{hexagonal sub-graph of} $\mathscr{H}$ (an example may be seen in Figure~\ref{fig:HexSub}, left). In order to count the number of perfect matchings of $G_{a,b,c}$ let us attach a weight of $1$ to each edge contained within it.

Suppose we identify together the edges of the plane and endow the (outer) surface of the resulting sphere with a sense of rotation in the clockwise direction. If the edges of the lattice are directed from black vertices to white then clearly within each hexagonal face of $G_{a,b,c}$ the direction of an odd number of edges will agree with the orientation of the plane when viewed from the centre of the face. Once we have convinced ourselves that the outer boundary (which is also a face) also satisfies this condition we see that this orientation of $G_{a,b,c}$ is already admissible (see Figure~\ref{fig:HexSub}, right).

If we label the $(ab+bc+ca)$-many vertices in each colour class that comprise $G_{a,b,c}$ then according to Kasteleyn's method the number of tilings of $G_{a,b,c}$ (denoted $M(G_{a,b,c})$) is
$$|\det(A_{G_{a,b,c}})|,$$
where $A_{G_{a,b,c}}=((A_{G_{a,b,c}})_{b_i,w_j})_{b_i,w_j\in G_{a,b,c}}$ is the bi-adjacency matrix of $G_{a,b,c}$ with entries given by
$$(A_{G_{a,b,c}})_{b_i,w_j}:=\begin{cases}1 & b_i,w_j \textrm{ {\it adjacent}},\\ 0 & otherwise.\end{cases}$$

\begin{rmk}Since we are chiefly concerned with counting perfect matchings on such hexagonal sub-graphs of $\mathscr{H}$ we shall abuse our notation in the following way: for general $a,b,c$ the admissibly oriented hexagonal sub-graph $G_{a,b,c}$ with edge weights of $1$ shall from now on be denoted $G$, and $A_{G}$ shall denote its corresponding bi-adjacency matrix.
\end{rmk}
\begin{figure}[t!]
	\includegraphics[width=0.4 \textwidth]{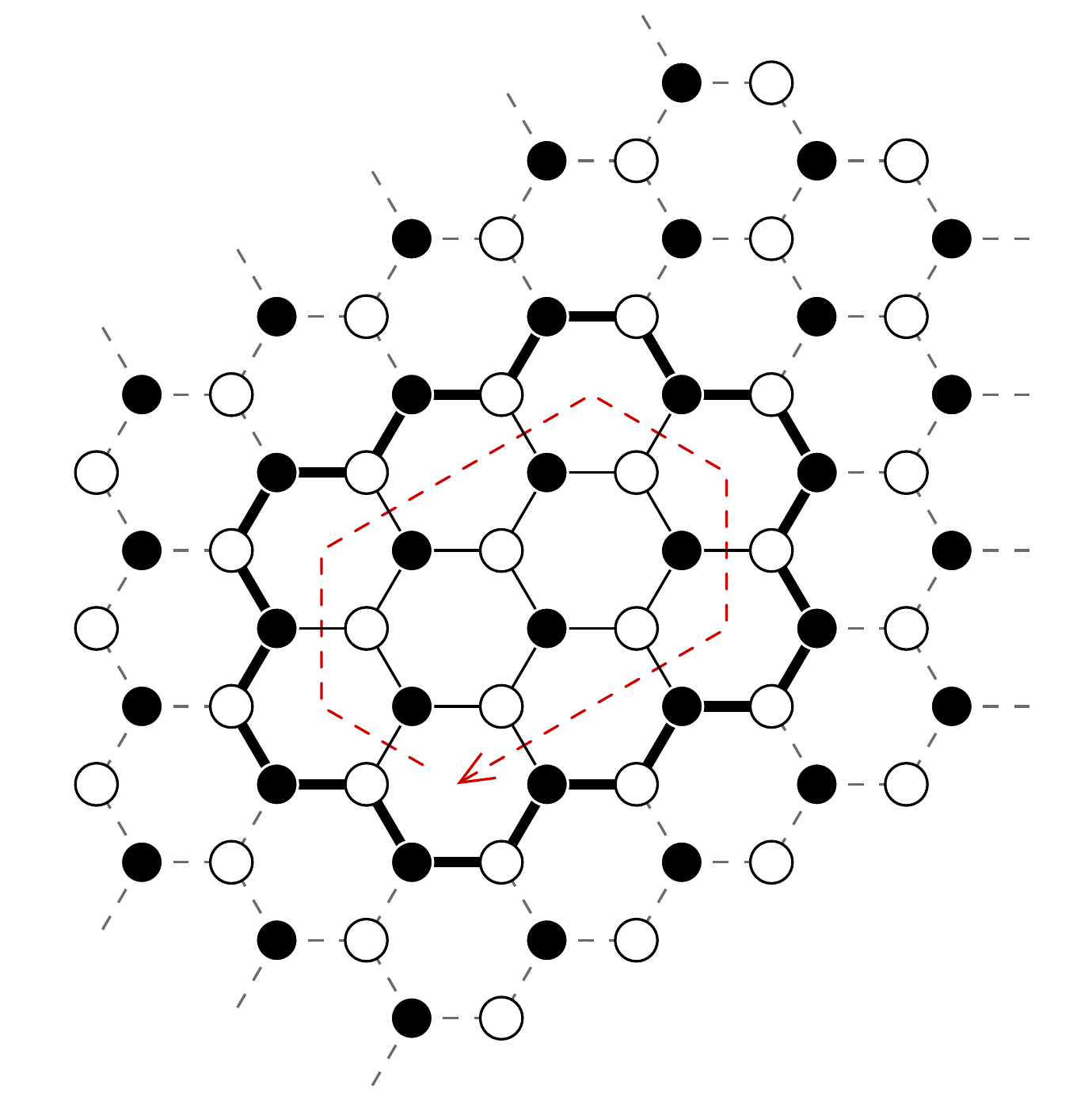}\,\,\,\,\,\,\,\,\includegraphics[width=0.4 \textwidth]{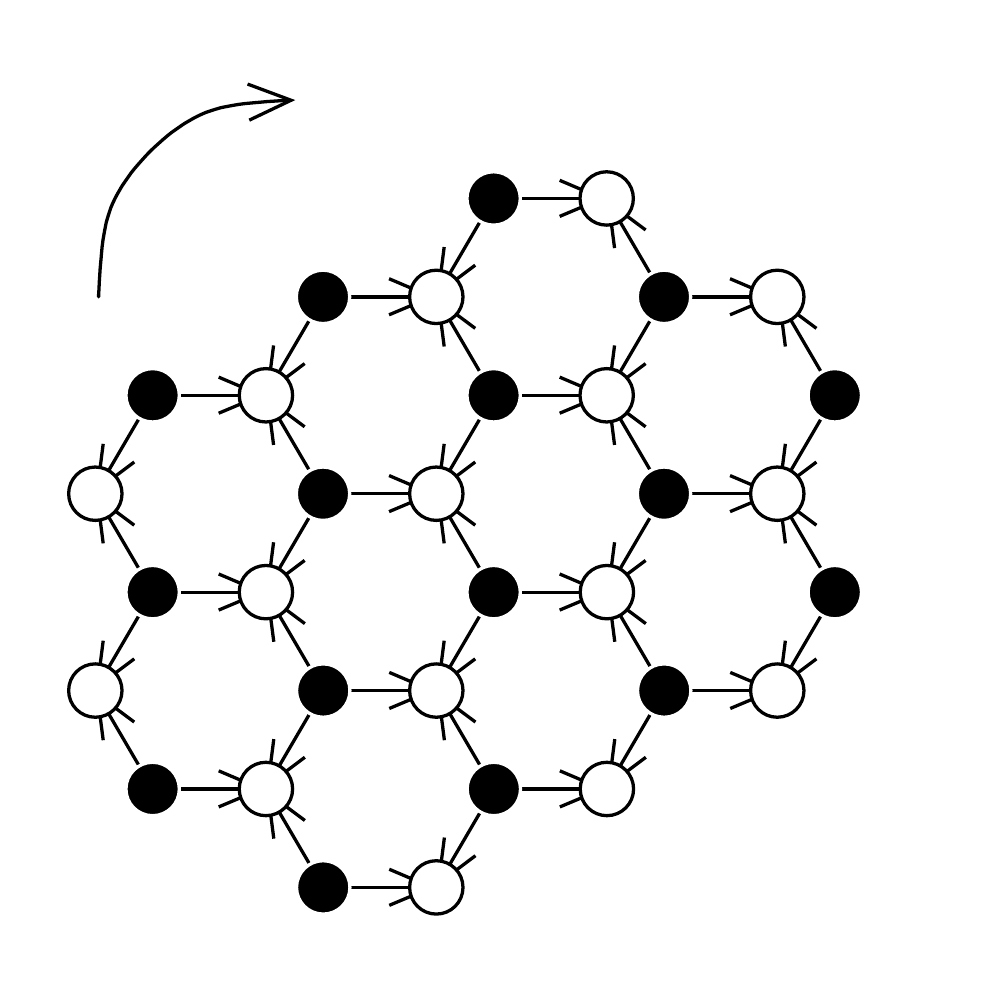}\caption{The hexagonal sub-graph $G_{2,2,3}\subset\mathscr{H}$ obtained by traversing the dotted path in red (left), together with the admissible orientation arising from directing its edges from black vertices to white (right).}\label{fig:HexSub}
\end{figure}
\subsection{Kasteleyn's method for sub-graphs with interior vertices removed}\label{subsec:KastHoles}
Consider two single vertices within $G$. If they lie on the same face then we say that they are \emph{connected via a face} (otherwise they are deemed to be \emph{unconnected}). Let $V:=\cup_{i}V_i$ be the set consisting of $k$-many black and $k$-many white vertices in $G$, where each $V_i$ is a \emph{connected} set of vertices (by which we mean either $V_i$ consists of a single vertex or for any $v\in V_i$, there exists at least one other $v'\in V_i, v'\neq v$ such that $v$ and $v'$ are connected via a face). Further to this we suppose that for any $v\in V_i$ and $v'\in V_j, j\neq i$, $v$ and $v'$ are unconnected. The set $V$ is thus an \emph{unconnected union} of connected sets of vertices. By removing $V$ from $G$ (together with all edges incident to those vertices in $V$) we obtain an hexagonal sub-graph of $\mathscr{H}$ that contains a set of unconnected gaps in its interior. We denote such a region $G\setminus V$ (see Figure~\ref{fig:AdmisPres}, centre and right). 

A natural question that now arises is whether the orientation of the edges that remain in $G\setminus V$ is again admissible. Suppose $V$ consists of a single vertex. Removing this vertex yields an oriented face in $G\setminus V$ that is not admissible, however it is easy to see that by removing a pair of vertices connected via a face we obtain a graph with gaps that is again admissibly oriented. This argument may be extended to larger sets of vertices, thus it follows that if $V$ is an unconnected union of connected sets of vertices where the parity of the number of white and black vertices in each connected set is the same, then $G\setminus V$ is admissibly oriented and we call $V$ an \emph{admissibility preserving} set of vertices (see Figure~\ref{fig:AdmisPres}).

\begin{rmk}\label{rmk:Induc}
	Suppose we remove a set of vertices $V$ from $G$. It is quite possible that in doing so, some subset of the vertices (say, $V'$) that remain in $G$ have matchings that are forced between them. In such a situation we may as well remove these extra vertices entirely, as they are fixed in every single matching of the remaining graph. If the set of \emph{induced holes} $V\cup V'$ is admissibility preserving then we call $V$ an \emph{admissibility inducing} set of vertices\footnote{It should be clear that in order for a matching of $G\setminus V$ to exist it must be the case that $V$ consists of an equinumerous number of white and black vertices.}.
\end{rmk}

\begin{lem}\label{lem:AdmissPres}
	For an hexagonal graph $G$ and an admissibility inducing set of vertices $V$ contained in its interior
	$$M(G\setminus V)=|\det(A_{G\setminus V})|,$$
	where $A_{G\setminus V}$ is the bi-adjacency matrix of $G\setminus V$.
\end{lem}
\begin{rmk}\label{rmk:Mat}
	The bi-adjacency matrix $A_{G\setminus V}$ is obtained by simply deleting from $A_G$ those rows and columns indexed by the black and white vertices in $V$ (respectively).
\end{rmk}
\begin{figure}[t!]
	\includegraphics[width=0.33 \textwidth]{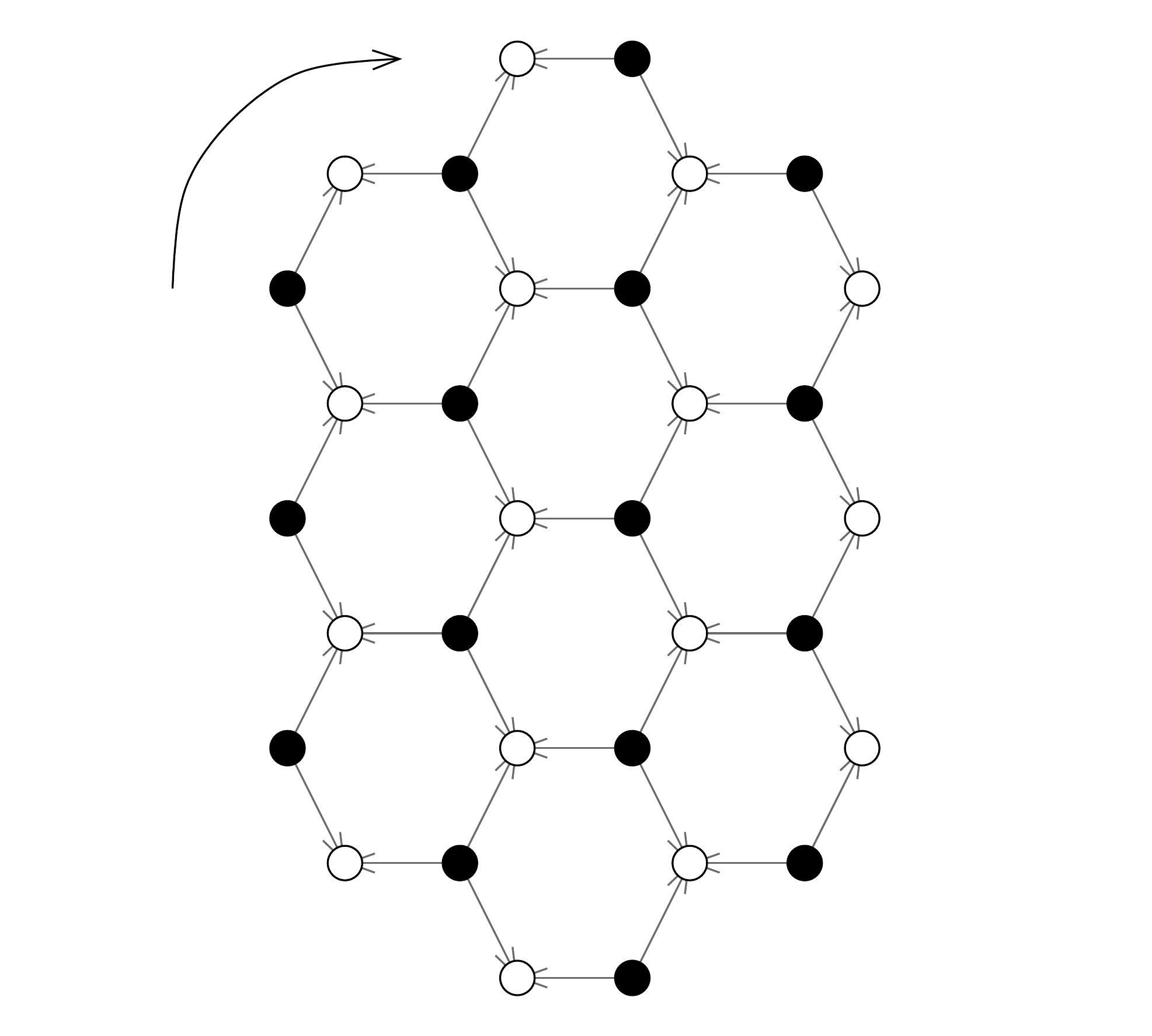}\includegraphics[width=0.33 \textwidth]{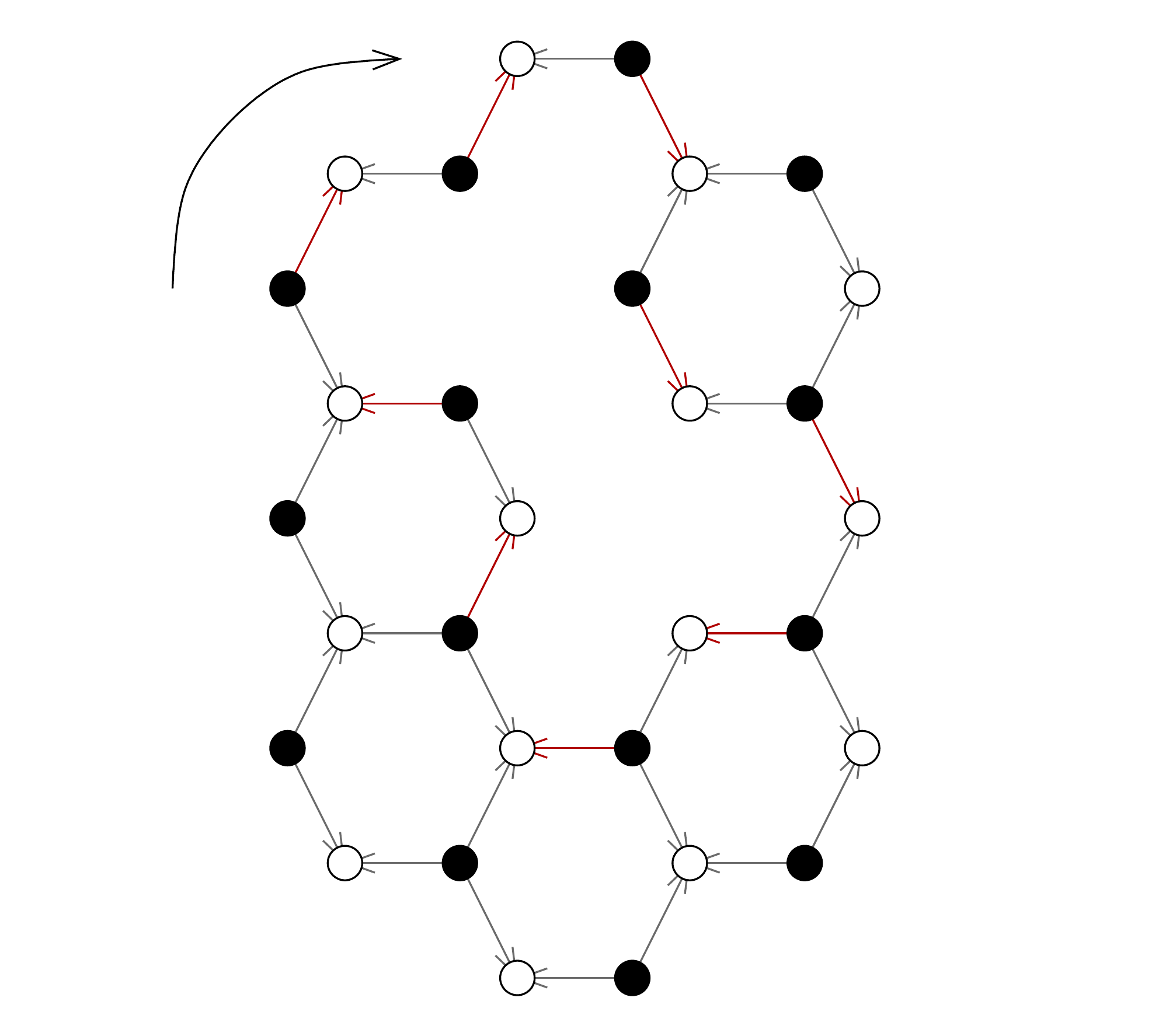}\includegraphics[width=0.33 \textwidth]{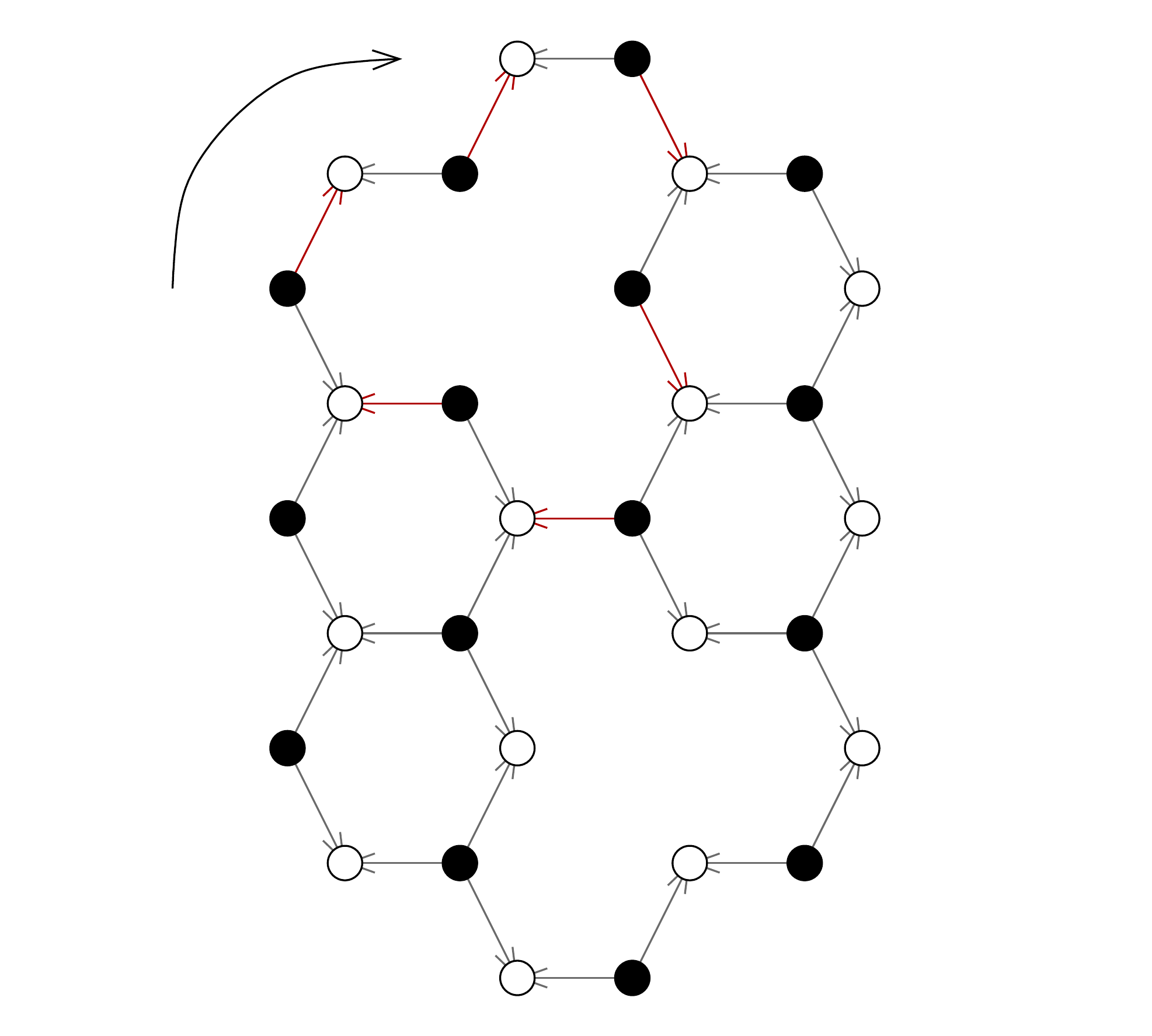}\caption{The admissibly oriented hexagonal sub-graph $G_{2,3,2}$ (left) from which a pair of admissibility preserving vertices have been removed (centre), and the same sub-graph with vertices removed that do not preserve the admissibility of its orientation.}\label{fig:AdmisPres}
\end{figure}
Our goal is to give an expression for $M(G\setminus V)$ that contains a determinant whose size is dependent on the size of the set $V$. Clearly if $|V|=2k$ (that is, $V$ consists of $k$-many black and $k$-many white vertices) then the expression in Lemma~\ref{lem:AdmissPres} gives $M(G\setminus V)$ as a determinant evaluation of a matrix whose size is dependent on the number of vertices that remain in $G\setminus V$, rather than those that have been removed\footnote{$A_{G\setminus V}$ is an $(ab+bc+ca-k)\times(ab+bc+ca-k)$ matrix.}. 

In his notes on dimer statistics Kenyon~\cite[Theorem~6]{Kenyon97} gives an alternative way of evaluating the determinant from the previous lemma.
\begin{thm*}[Kenyon]
	Suppose $V$ is a set of vertices contained in $G$ and let $(A_{G}^{-1})_V$ denote the sub-matrix obtained by restricting the inverse of $A_G$ to those rows and columns indexed by the white and black (respectively) vertices in $V$. Then
	$$|\det(A_{G}\setminus V)|=|\det(A_G)\cdot\det((A_{G}^{-1})_V)|.$$
\end{thm*}
\begin{rmk}
	Observe that the above theorem holds irrespective of whether $V$ is admissibility inducing or not.
\end{rmk}
We already have a closed form evaluation for $|\det(A_G)|$, it is the well-known and celebrated formula due to MacMahon~\cite{MacMahon16}
$$|\det(A_G)|=\prod_{i=1}^a\prod_{j=1}^b\prod_{k=1}^c\frac{i+j+k-1}{i+j+k-2},$$
thus what remains is to determine the entries of the matrix $A_{G}^{-1}$ (referred to as the \emph{inverse Kasteleyn matrix of} $G$). 
\begin{rmk}\label{rmk:MacM}
	It should be noted that MacMahon's original formula came not from considering tilings, but instead arose from his interest in enumerating \emph{boxed plane partitions}. The bijection that exists between plane partitions that fit inside an $a\times b\times c$ box and rhombus tilings of $H$ is quite beautiful, however we shall not discuss it here.
\end{rmk}
If $b_j$ and $w_i$ are two vertices in $V$ then according to Cramer's rule the $(w_i,b_j)$-entry of $A_{G}^{-1}$ is
$$(-1)^{i+j}\cdot\frac{\det(A_{G\setminus\{b_j,w_i\}})}{\det(A_G)}.$$
 
For a pair of vertices $b_j,w_i$ that are admissibility inducing, the graph $G\setminus \{b_j,w_i\}$ is admissibly oriented and the numerator above gives (up to sign) $M(G\setminus \{b_j,w_i\})$. If, however, $b_j$ and $w_i$ are unconnected then $G\setminus\{b_j,w_i\}$ is not admissibly oriented, and so $|\det(A_{G\setminus\{b_j,w_i\}})|$ counts instead the number of \emph{signed perfect matchings\footnote{The determinant is a sum over perfect matchings where each summand has a certain sign, this gives rise to the term.} of $G$} (see Cook and Nagel~\cite{CookNagel15} for further details).
\begin{figure}[t!]
	\includegraphics[scale=0.5]{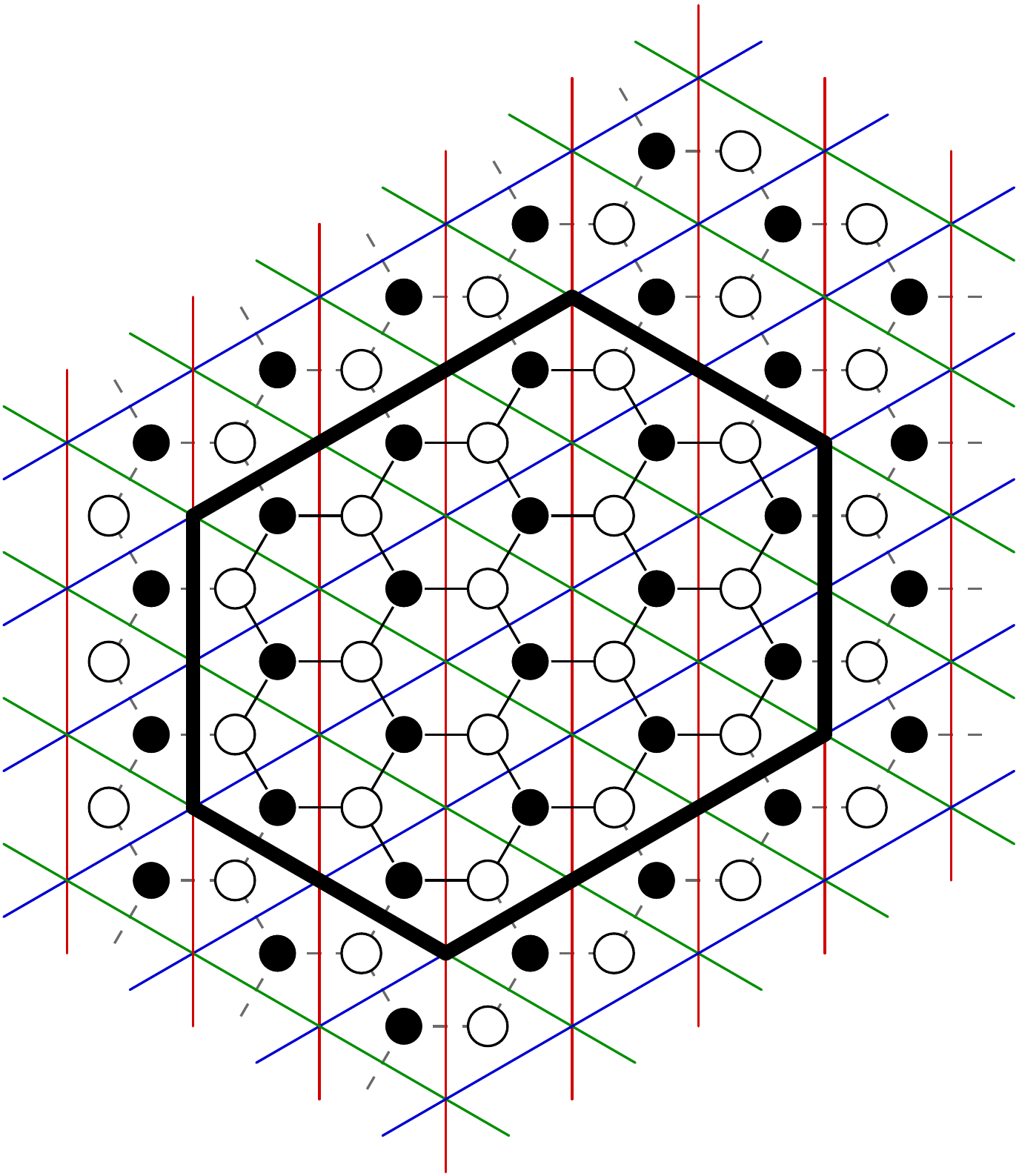}\caption{The ``dual" of the hexagonal lattice is given by the coloured lines: those that are green belong to $L_-$; those that are blue to $L_+$; and those that are red to $L_{\infty}$. The region corresponding to $G_{2,2,3}\subset\mathscr{H}$ is the hexagon $H_{2,2,3}\subset\mathscr{T}$ outlined in black.}\label{fig:Dual}
\end{figure}
If we remain on the hexagonal lattice, viewing this problem purely from the stand-point of perfect matchings, the way forward appears somewhat murky. Calculating an entry of the inverse Kasteleyn matrix of $G$ involves taking the determinant of a sub-matrix obtained by deleting a single row and column from $A_G$, furthermore each entry corresponds to removing a unique row and column combination. It turns out, however, that the determinant of $A_{G\setminus\{b_j,w_i\}}$ may be replaced with the determinant of a so-called \emph{lattice path matrix} (see Section~\ref{sec:RhombTil}) that arises from translating perfect matchings of $G\setminus V$ into rhombus tilings of holey hexagons, which are in turn translated into families of non-intersecting lattice paths. Under this replacement the different entries of $A_{G}^{-1}$ are computed by taking the determinant of lattice path matrices that differ only in their last row and column; we shall soon see how this simplifies our task of finding a closed form expression for $\det(A_{G\setminus\{b_j,w_i\}})$.

\section{Rhombus tilings on the triangular lattice and families of non-intersecting paths}\label{sec:RhombTil}

Let us return to the tiling of the plane by regular hexagons discussed at the beginning of Section~\ref{subsec:KastHex}. Imagine we place a point at the centre of each hexagonal face and join with a straight line all pairs of points that are located within two different faces that share a common edge. Once the hexagonal tiles have been removed what remains is a tiling of the plane by right and left pointing unit equilateral triangles, which is known as the \emph{unit triangular lattice}\footnote{This is sometimes referred to as the \emph{dual} of $\mathscr{H}$.} and shall be denoted $\mathscr{T}$. This consists of three (infinite) families of lines $L_+, L_{-}, L_{\infty}$, where in each family all lines have the same gradient: $L_+$ consists of a set of lines in the polar direction $\pi/6$, separated by a distance of $\sqrt{3}$ in the horizontal direction; $L_-$ is the family of lines in the polar direction $-\pi/6$, separated by a unit distance along the lines in $L_+$; and $L_{\infty}$ consists of a family of vertical lines that intersect all points where the lines in $L_-$ and $L_+$ intersect (see Figure~\ref{fig:Dual}). 

Under this construction the set of all vertices that are in the same colour class in $\mathscr{H}$ correspond to the set of all unit triangles on $\mathscr{T}$ that point in one direction, so without loss of generality we may assume that black vertices in $\mathscr{H}$ correspond to left pointing unit triangles in $\mathscr{T}$. Furthermore the region $G\subset\mathscr{H}$ corresponds to a semi-regular hexagon with sides of length $a,b,c,a,b,c$ (going clockwise from the south-west side) on $\mathscr{T}$. For specific $a,b,c$ we shall denote such a region $H_{a,b,c}$, otherwise in the general case we shall denote it simply by $H$. It follows that a matching between a black and a white vertex in $\mathscr{H}$ corresponds to joining together a pair of unit triangles (one left pointing, one right pointing) that share precisely one edge in $\mathscr{T}$ (hence forming a unit rhombus), thus perfect matchings of $G\setminus V$ are in bijection with \emph{rhombus tilings\footnote{We shall often refer to rhombus tilings simply as \emph{tilings}.} of $H\setminus T$}, where $T$ is the set of unit triangles corresponding to the vertices in $V$ (see Figure~\ref{fig:HnT}).
\begin{figure}[t!]
	\includegraphics[width=0.33 \textwidth]{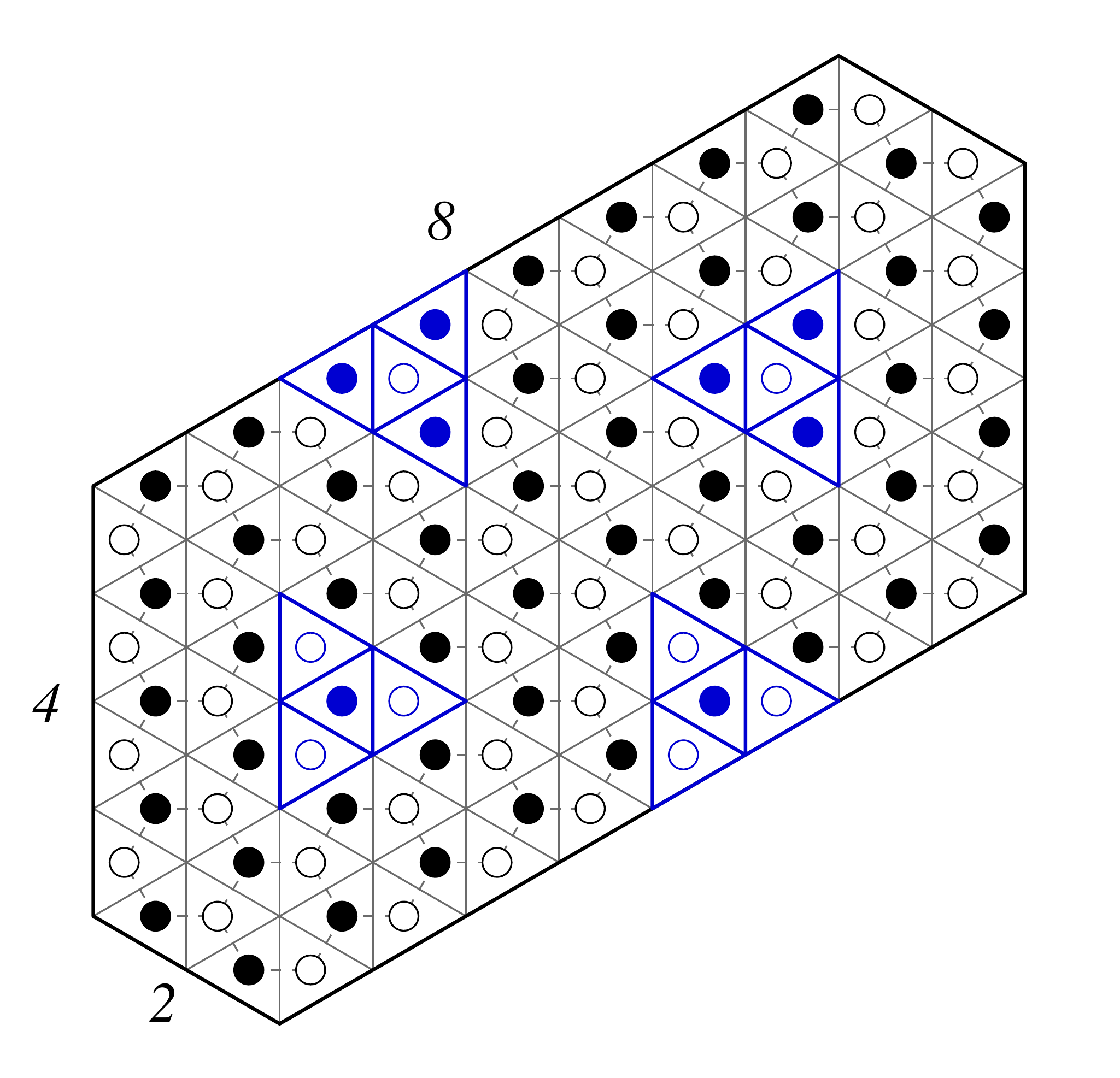}\includegraphics[width=0.33 \textwidth]{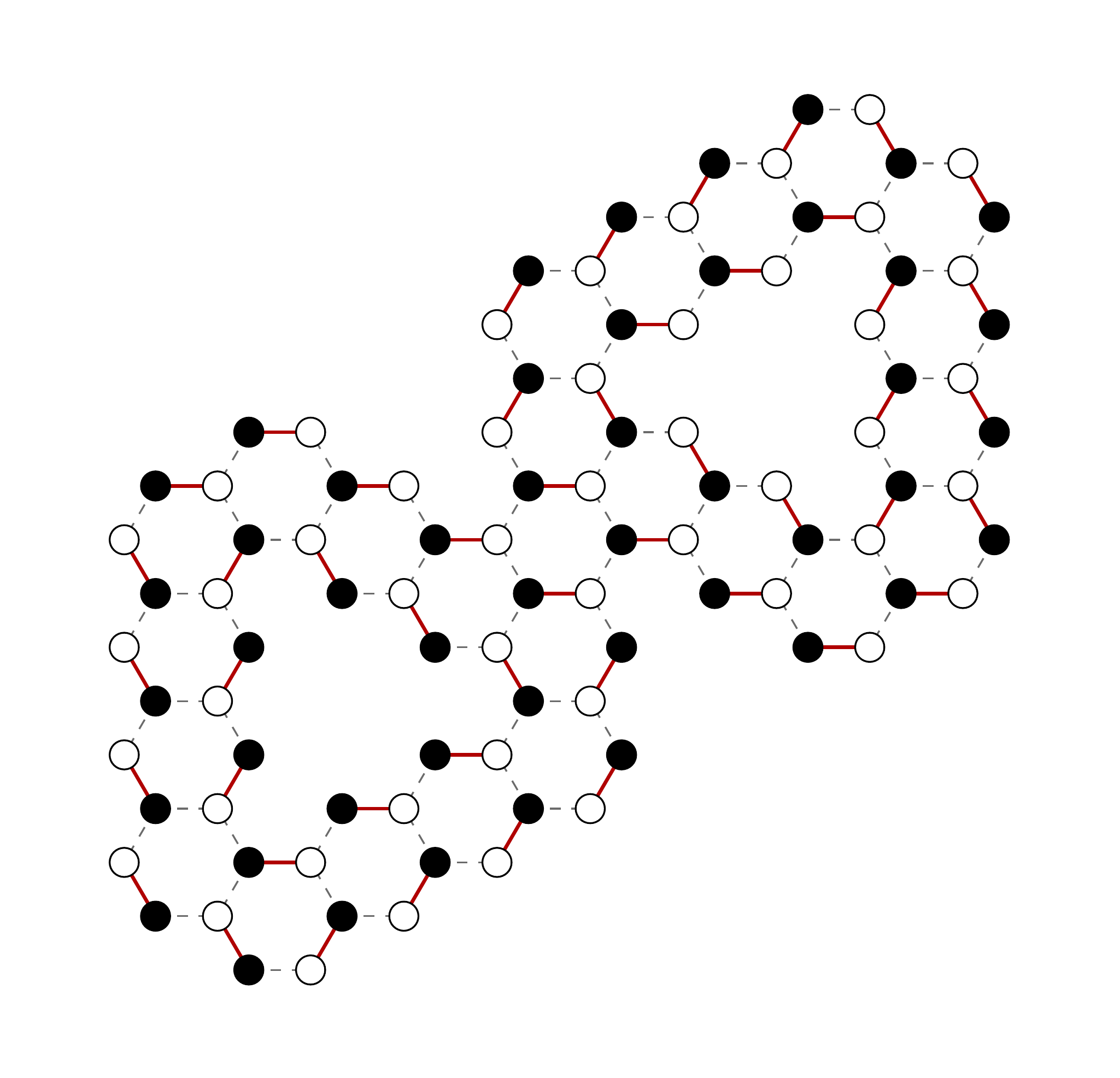}\includegraphics[width=0.33 \textwidth]{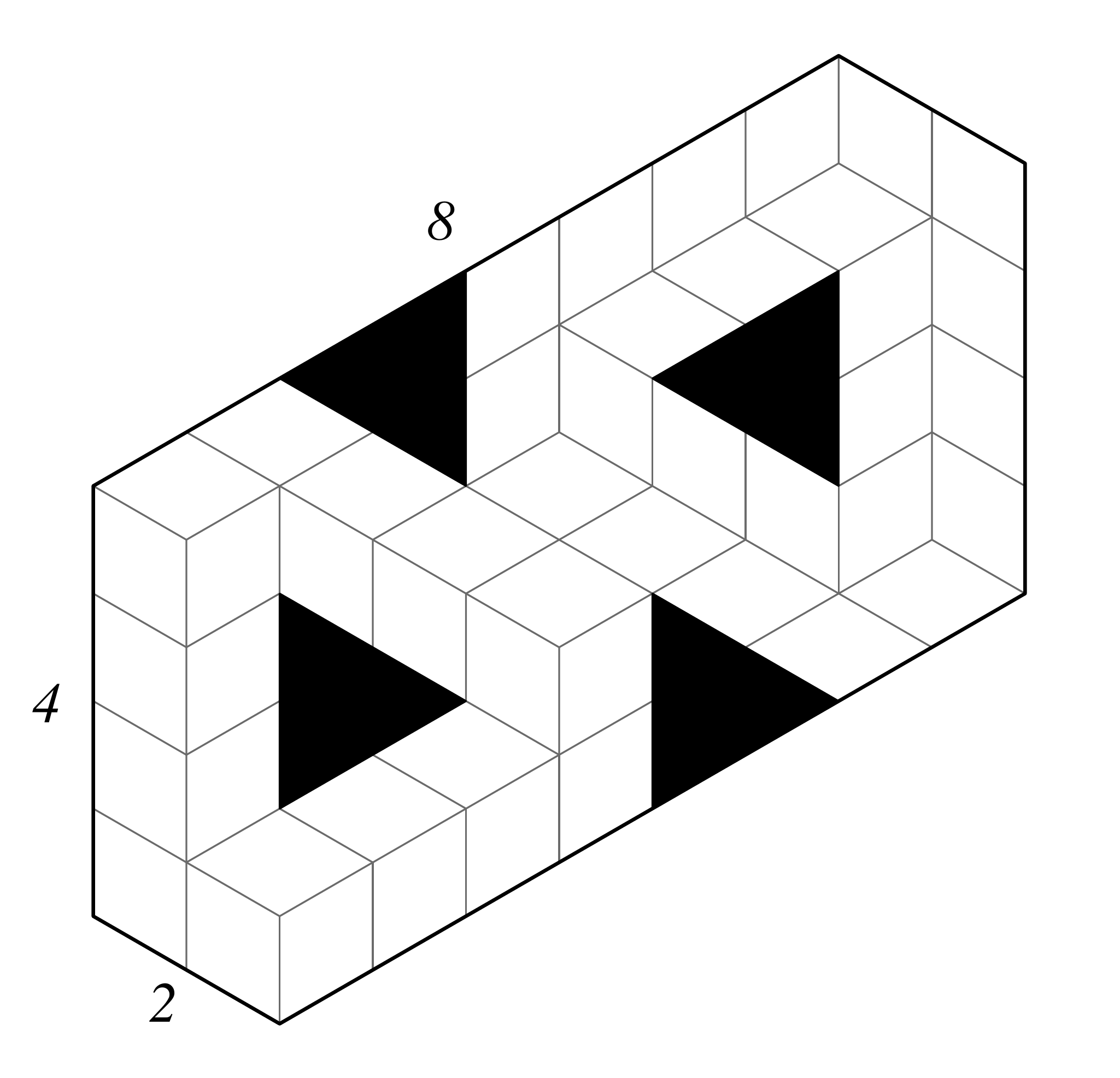}\caption{From left to right: the region $G_{2,4,8}\setminus V\subset\mathscr{H}$ where the admissibility preserving vertices $V$ are coloured blue (together with the corresponding sub-region of $\mathscr{T}$); a perfect matching of $G_{2,4,8}\setminus V$; the corresponding rhombus tiling of $H_{2,4,8}\setminus T$.}\label{fig:HnT}
\end{figure}
\begin{rmk}\label{rmk:charge}
	It should be observed that for an unconnected union of connected sets of vertices $V:=\cup_iV_i$, each subset $V_i\subset\mathscr{H}$ corresponds to a set $T_i$ of unit triangles in $\mathscr{T}$ that are \emph{connected via points or edges} (by which we mean that for any triangle $t\in T_i$ there exists at least one other triangle $t'\in T_i, t'\neq t,$ such that $t$ and $t'$ share an edge or touch at a point) and furthermore no triangle $t\in T_i$ is connected via an edge or a point to a triangle $t'\in T_j$, for $i\neq j$. It should be plain to see that if $V$ is admissibility preserving then the number of left and right pointing triangles in each of the sets corresponding to the $V_i$s is even, hence we shall refer to a set $T:=\cup_{i}T_i$ as a set of holes of even \emph{charge}\footnote{The \emph{charge} of a hole $T_i\in T$, denoted $q(T_i)$, is the difference between the number of right and left pointing unit triangles that comprise it.}. 
\end{rmk}
\begin{rmk}
	As in Remark~\ref{rmk:Induc}, a set of unit triangular holes $T$ may give rise to forced rhombi in tilings of $H\setminus T$. If we denote the unit triangles that comprise these rhombi $T'$, then the number of tilings of $H\setminus T$ is equal to the number of tilings of $H\setminus (T\cup T')$. If the set $T\cup T'$ corresponds to a set of holes of even charge then we say that $T$ is an \emph{even charge inducing} set of holes.
\end{rmk}

By considering perfect matchings of sub-graphs of $\mathscr{H}$ in terms of their equivalent representations on $\mathscr{T}$ we are afforded an entirely different perspective from which we may view tilings of $H\setminus T$. Within the folklore of the theory of plane partitions and rhombus tilings there exists a bijection that allows one to represent tilings of sub-regions of $H\setminus T$ as families of non-intersecting lattice paths. We recall this bijection in the following sections.

\subsection{A classical bijection}\label{subsec:ClassBij}

Take a rhombus tiling of $H\setminus T$ and place start (end, respectively) points at the mid-points of the south-west (north-east) side of each unit rhombus that lies along the south-west (north-east) edge. Apply the same procedure to those rhombi that lie along the north-east (south-west) edges of any holes that lie within its interior. We label the set of start points $S$ and the set of end points $E$.
\begin{figure}[t!]
	\includegraphics[width=0.4 \textwidth]{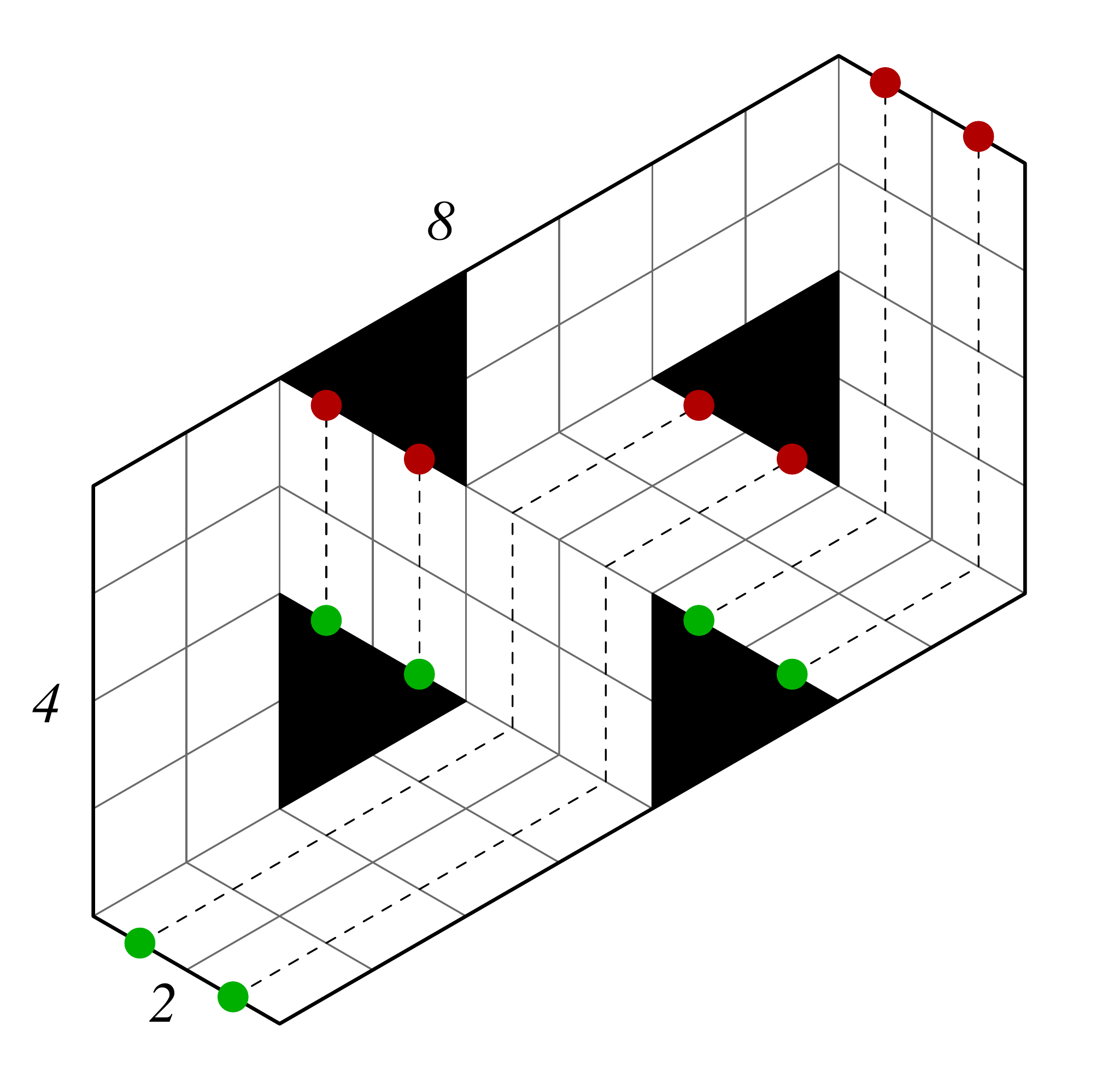}\,\,\,\,\,\,\,\,\,\,\,\,\includegraphics[width=0.4 \textwidth]{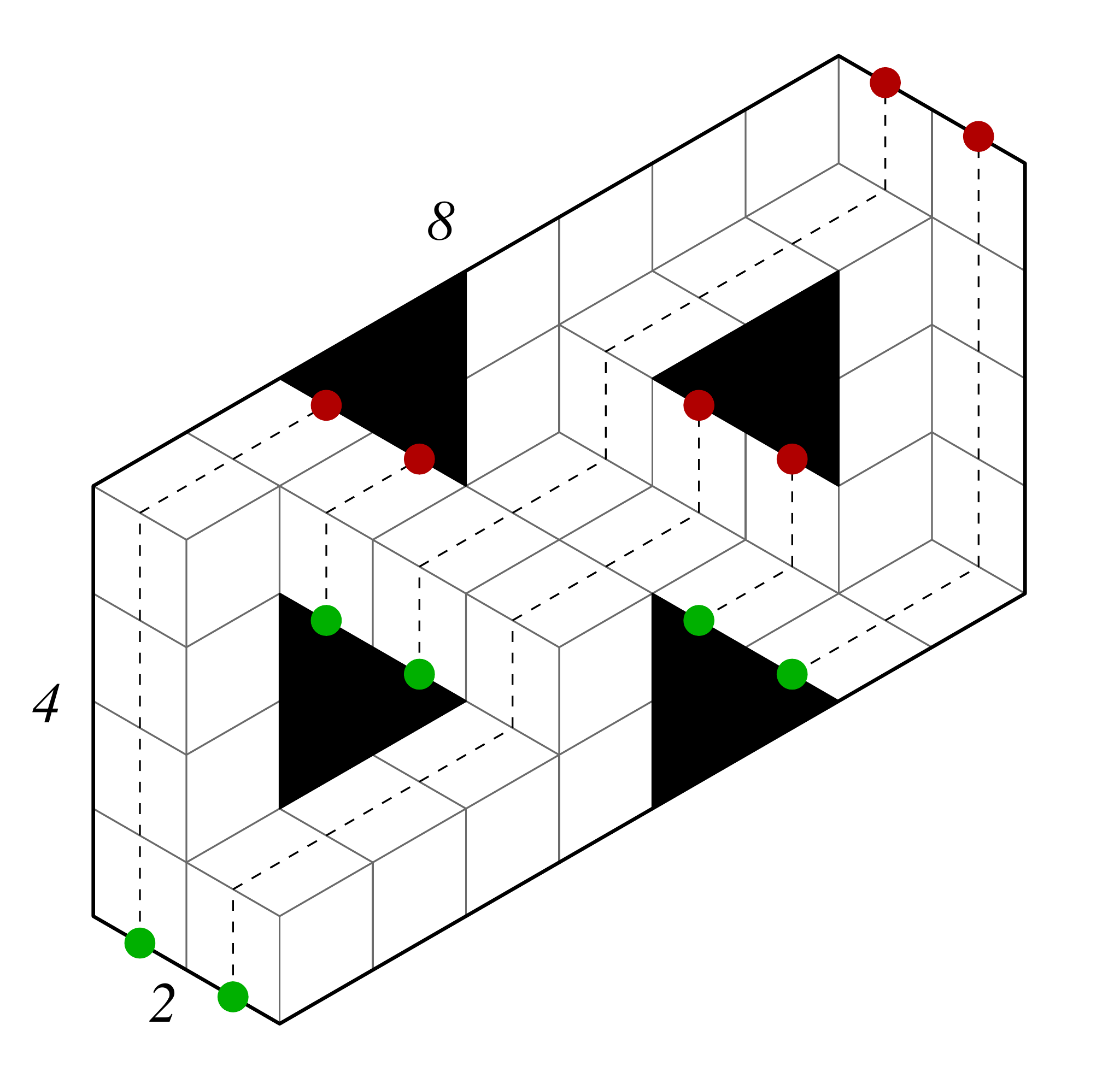}\caption{Two different tilings of $H_{2,4,8}\setminus T$ and the corresponding families of paths across rhombi (the start points are green, the end points red), where $T$ is a set of holes of even charge.}\label{fig:TwoTil}
\end{figure}

From a start point $s\in S$ we may construct a \emph{path across unit rhombi} by travelling from one side of a rhombus to its opposite parallel side, and then repeating this process across every rhombus we encounter until our path meets with some end point $e\in E$. By constructing such a path for every start point in $S$ we obtain a \emph{family of non-intersecting paths across unit rhombi}\footnote{Within this context \emph{non-intersecting} means that no two paths traverse a common rhombus.} that correspond to a particular rhombus tiling of $H\setminus T$. It follows that the set of rhombus tilings of $H\setminus T$ may be represented as a set of families of non-intersecting paths across unit rhombi, where every path traverses rhombi that are oriented in one of two ways (see Figure~\ref{fig:TwoTil}). Moreover it is easy to see that every rhombus contained in $H\setminus T$ that is oriented in one of these two directions is traversed by such a path and so a family of paths beginning at $S$ and ending at $E$ determines a tiling completely. These paths across rhombi may in turn be translated into non-intersecting lattice paths consisting of unit north and east steps on $\mathbb{Z}_{a,c}\times\mathbb{Z}_{a,b}$ (where $\mathbb{Z}_{p,q}$ denotes the set $\{x+y/2:x\in\mathbb{Z},y=p+q-1 \text{ (mod }2)\}$), however in order to state this bijection explicitly we must first introduce some notation so that we can specify each unit triangle contained in $H\setminus T$.

\subsection{Labelling the interior of the hexagon}

Consider the hexagonal region $H\setminus T$. We may place an origin $O$ at its centre, that is, at the intersection of the pair of lines that intersect the mid-points of two distinct pairs of parallel sides of $H\setminus T$ (for example, let $l_a,l_b$ be the lines intersecting the mid-points of the sides of length $a,b$ respectively and place $O$ at the intersection of $l_a$ and $l_b$). Let $h$ denote the horizontal line that intersects $O$. We proceed by labelling the lines in each of the families $L_+,L_-$, and $L_{\infty}$ according to their distance and location with respect to $O$ along $h$. Every line $l$ that intersects $h$ lies at a distance of $l_d\cdot(\sqrt{3}/2)$ from $O$, for some $l_d\in\mathbb{Z}/2$ (this lattice distance, $l_d$, is negative if the intersection lies to the left of $O$, positive if it lies to the right). We label each line $l\in L_-$ or $L_+$ with $l_d/2$, otherwise we label $l$ with $l_d$. The region $H\setminus T$ is thus the sub-region of $\mathscr{T}$ enclosed by the lines labelled $\pm\tfrac{b+c}{2}\in L_-$, $\pm\tfrac{a+b}{2}\in L_+$, and $\pm\tfrac{a+c}{2}\in L_{\infty}$. It follows that each triangle contained in $H\setminus T$ may be described by a triple $(l,l',l'')$ where $l\in\{-\frac{b+c}{2},1-\frac{b+c}{2},\dots,\frac{b+c}{2}\}$, $l'\in\{-\frac{a+b}{2},1-\frac{a+b}{2},\dots,\frac{a+b}{2}\}$, and $l''\in\{-\frac{a+c}{2},1-\frac{a+c}{2},\dots,\frac{a+c}{2}\}$ (see Figure~\ref{fig:LineLab}).

\begin{figure}[t]
	\includegraphics[width=0.4 \textwidth]{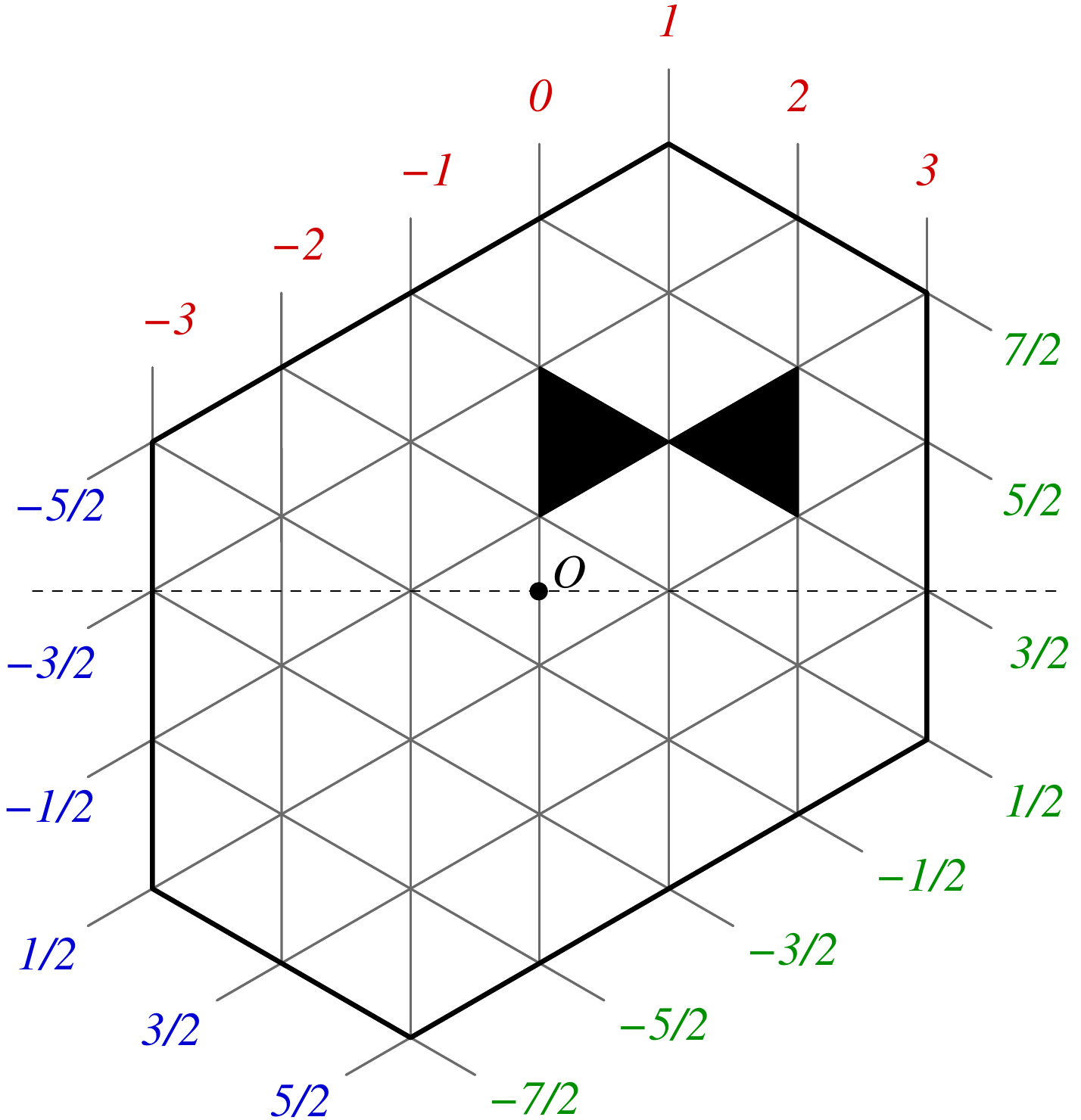}\,\,\,\,\,\,\,\,\,\,\,\,\,\,\,\,\,\,\,\,\includegraphics[width=0.4 \textwidth]{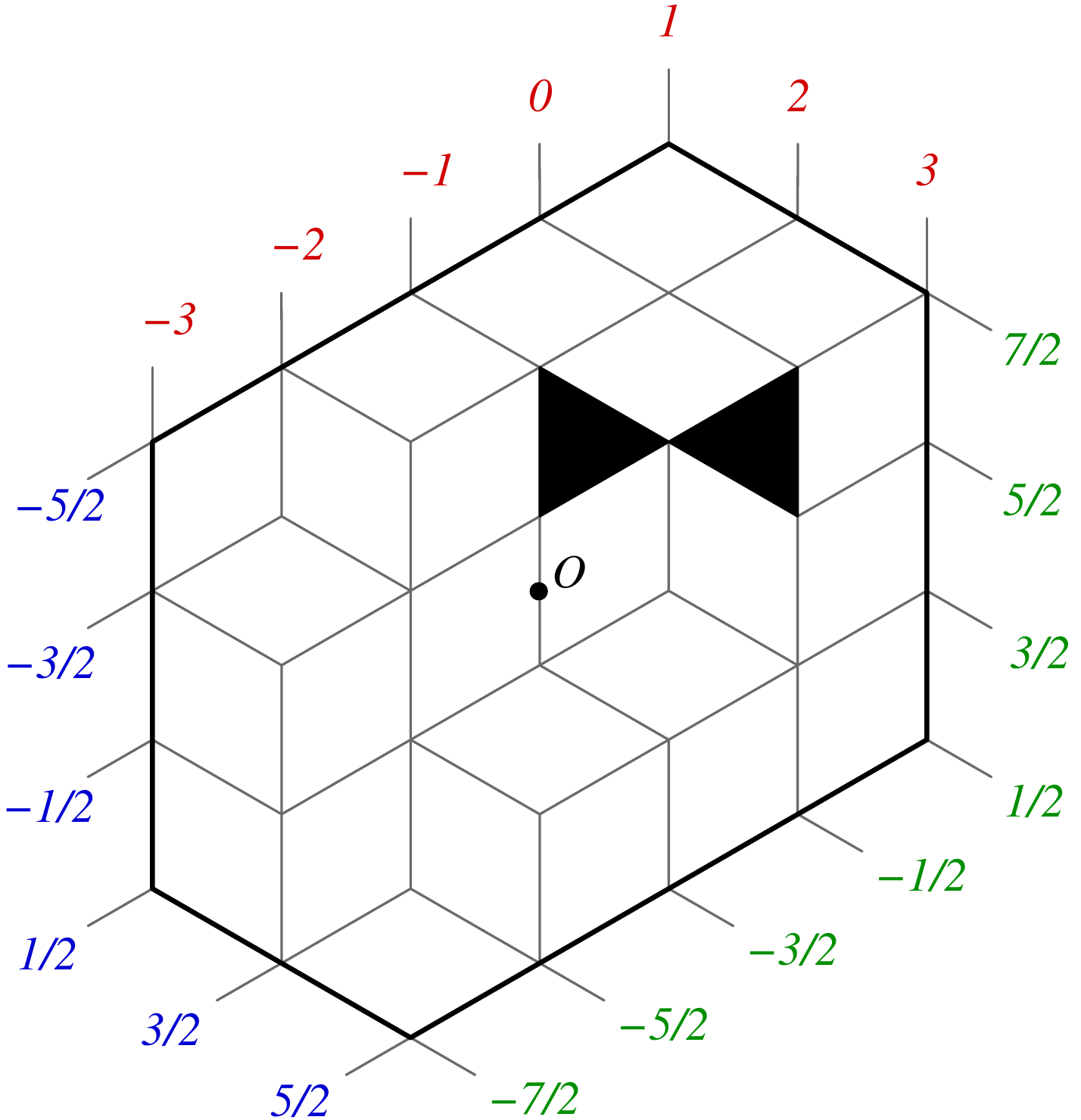}\caption{The hexagonal region $H_{2,3,4}\setminus\{(\tfrac{3}{2},-\tfrac{1}{2},0),(\tfrac{3}{2},-\tfrac{1}{2},2)\}$ (left) together with a tiling of the region obtained by joining together pairs of unit triangles contained within it that share an edge (right).}\label{fig:LineLab}
\end{figure}
\subsection{Translating rhombus tilings into families of non-intersecting paths}
We have already established that rhombus tilings of $H\setminus T$ give rise to families of non-intersecting paths across unit rhombi. According to our convention regarding start and end points the unit rhombi that are traversed by these paths are either \emph{horizontal} (by which we mean each one is formed by joining together the left pointing unit triangle $(l,l',l'')$ with $(l+1,l'+1,l'')$) or \emph{left leaning} (these are formed by joining together the left pointing unit triangle $(l,l',l'')$ with the right pointing $(l+1,l',l''-1)$)\footnote{The other type of rhombi contained in each tiling shall be referred to as \emph{right leaning} and are formed by joining a left pointing triangle $(l,l',l'')$ with a right pointing one $(l,l'+1,l''-1)$.}.

Let us identify the set of start points of our paths across rhombi by the left pointing unit triangles on whose south-west edges these start points lie, thus $S:=S_H\cup S_T$ where
$$S_H:=\{(-\tfrac{b+c}{2},\tfrac{b-a}{2}+i-1,i-\tfrac{a+c}{2})\in H\setminus T:1\leq i\leq a\}$$
denotes the set of triangles that lie along the south-west edge of $H\setminus T$ and $$S_T:=\{(l,l',l'')\in H\setminus (T\cup S_H):(l,l'+1,l''-1)\notin H\setminus T\}$$ those that lie along the north-east edge of any holes in its interior.

In a similar way we shall identify the end points $E$ with the right pointing unit triangles on whose north-east edges the end points lie, thus $E:=E_H\cup E_T$ where
$$E_H:=\{(\tfrac{b+c}{2},j-\tfrac{a+b}{2},\tfrac{c-a}{2}+j-1)\in H\setminus T:1\leq j\leq a\}$$
corresponds to those points in $E$ that lie along the north-east boundary of $H\setminus T$ and
$$E_T:= \{(l,l',l'')\in H\setminus (T\cup E_H):(l,l'-1,l''+1)\notin H\setminus T\}$$ are the unit triangles corresponding to the points in $E$ that lie along the south-west edge of any holes in its interior.
\begin{figure}
	\includegraphics[width=0.4 \textwidth]{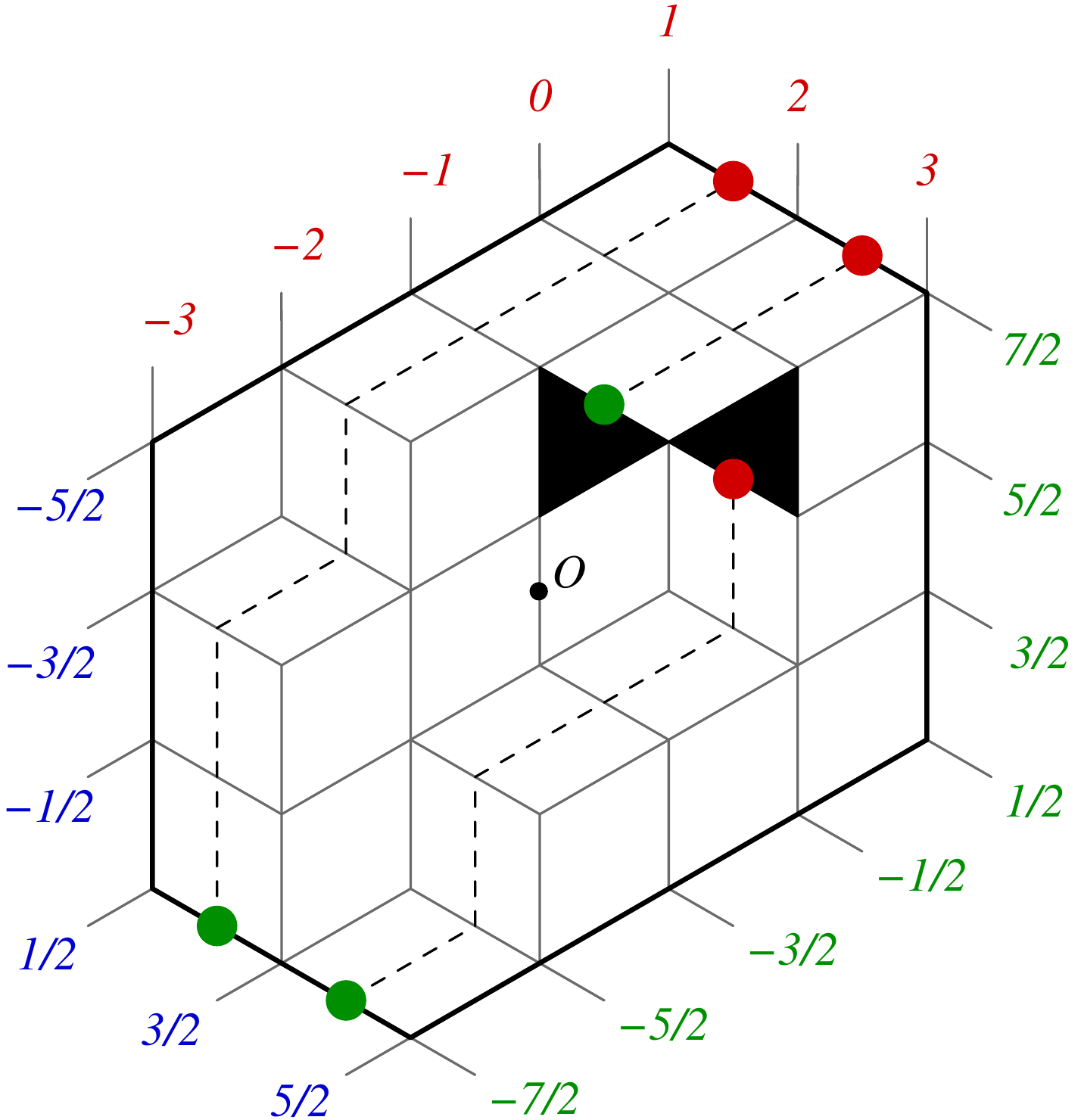}\,\,\,\,\,\,\,\,\,\,\,\,\,\,\,\,\,\,\,\,\includegraphics[width=0.4 \textwidth]{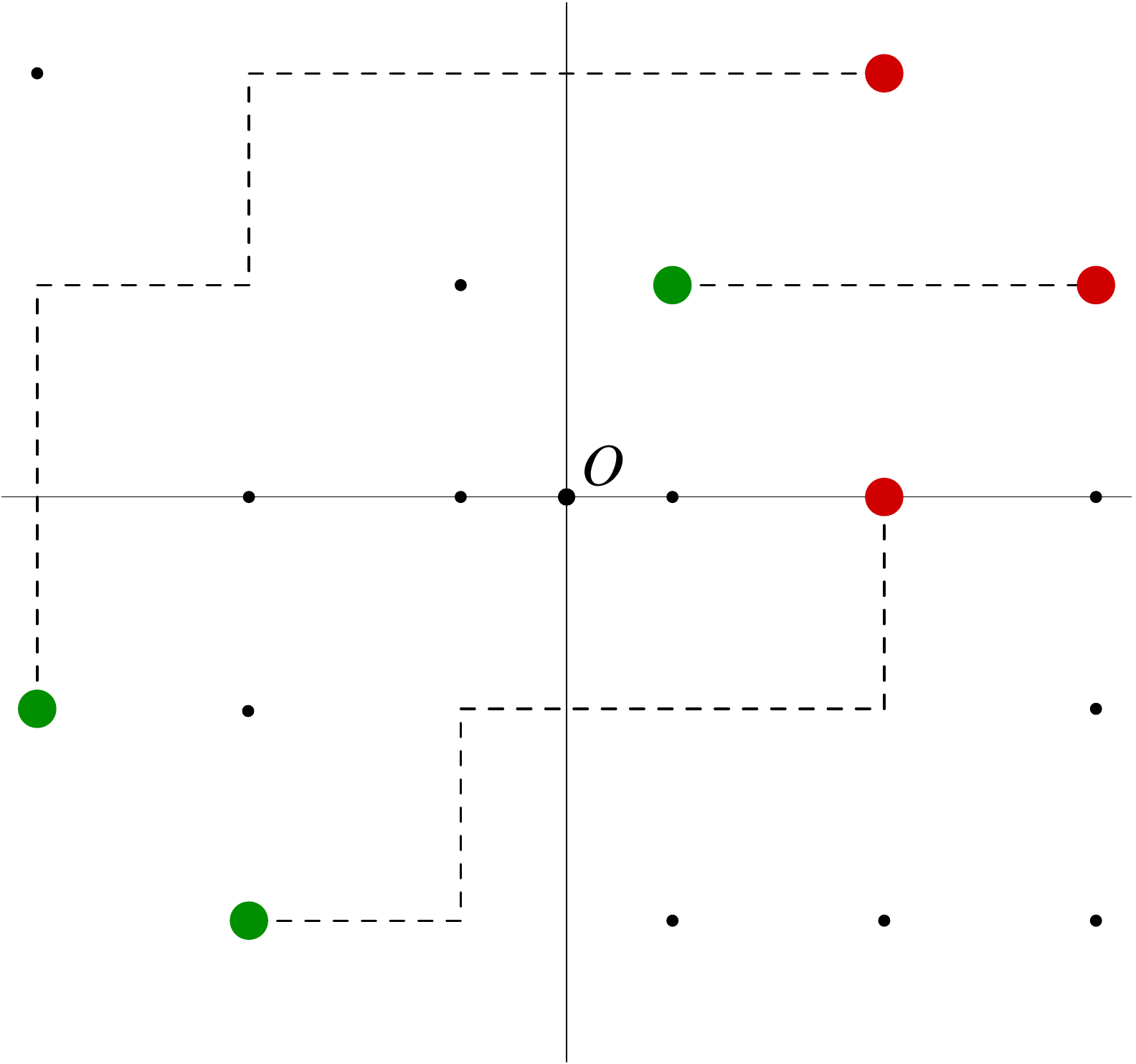}\caption{A family of non-intersecting lattice paths across unit rhombi (left), and the translation of these paths into non-intersecting lattice paths on $(\mathbb{Z}+\tfrac{1}{2})\times\mathbb{Z}$ consisting of north and east unit steps (right).}\label{fig:Paths}
\end{figure}
A path across rhombi from a point in $S$ to a point in $E$ consisting of $p$-many horizontal and $q$-many left leaning rhombi may then be written as a tuple $(R_1,\dots,R_{p+q})$ of pairs of unit triangles $R_i:=(\triangleleft_{i},\triangleright_{i})$ corresponding to either left leaning or horizontal rhombi, where $\triangleleft_1\in S$, $\triangleright_{p+q}\in E$, and the north-east side of $R_i$ coincides with the south-west side of $R_{i+1}$ for $1\le i\le p+q-1$ (that is, the first coordinate of $\triangleright_i$ agrees with that of $\triangleleft_{i+1}$).

Consider now the function $\psi:\mathbb{Z}_{b,c}\times\mathbb{Z}_{a,b}\times\mathbb{Z}_{a,c}\rightarrow\mathbb{Z}_{a,c}\times\mathbb{Z}_{a,b}$ given by
$$\psi((x,y,z)):=(\tfrac{1}{2}(x+y+z),\tfrac{1}{2}(x-y-z)).$$
For a horizontal rhombus $R_i$ we have
$$\psi:((l,l',l''),(l+1,l'+1,l''))\rightarrow((\tfrac{1}{2}(l+l'+l''),\tfrac{1}{2}(l-l'-l'')),(\tfrac{1}{2}(l+l'+l'')+1,\tfrac{1}{2}(l-l'-l''))),$$
thus $\psi$ maps horizontal rhombi to a pair of coordinates in $\mathbb{Z}_{a,c}\times\mathbb{Z}_{a,b}$ that describe an east unit step beginning at $\psi(\triangleleft_i)$ and ending at $\psi(\triangleright_i)$. In a similar way it can be shown that if $R_i$ is instead a left leaning rhombus then $\psi$ maps $R_i$ to a pair of coordinates that describe a north unit step\footnote{Similarly $\psi$ maps a right leaning rhombus to a single point.}. Furthermore, for $\triangleright_i\in R_i$ and $\triangleleft_{i+1}\in R_{i+1}$, we have $\psi(\triangleright_{i})=\psi(\triangleleft_{i+1})$, hence under $\psi$ a path across rhombi corresponds to a sequence of coordinates that encode a \emph{lattice path} on $\mathbb{Z}_{a,c}\times\mathbb{Z}_{a,b}$ that begins at $\psi(\triangleleft_1)=(x,y)$, ends at $\psi(\triangleright_{p+q})=(x+p,y+q)$, and consists of $p$-many east and $q$-many north unit steps.

Applying $\psi$ to every path across rhombi obtained from a tiling of $H\setminus T$ yields a family of lattice paths beginning at $S_\psi:=\{\psi(\triangleleft):\triangleleft\in S\}$ and ending at $E_{\psi}:=\{\psi(\triangleright):\triangleright\in E\}$. Since in any tiling of $H\setminus T$ no two paths across rhombi will traverse a common rhombus, it follows that no two lattice paths in this family will intersect at a common vertex in $\mathbb{Z}_{a,c}\times\mathbb{Z}_{a,b}$. The number of tilings of $H\setminus T$ is then the number of families of non-intersecting lattice paths that begin at $S_{\psi}$ and end at $E_{\psi}$, and from now on we shall use $S$ and $E$ to denote these sets of points (respectively). An example of a tiling together with the corresponding non-intersecting lattice paths may be found in Figure~\ref{fig:Paths}.

\subsection{The lattice path matrix}\label{subsec:LattPathMat}
Consider now the two tuples of start and end points, say $S:=(s_1,s_2,\dots,s_{|S|})$ and $E:=(e_1,e_2,\dots,e_{|E|})$. We know that tilings of $H\setminus T$ correspond to non-intersecting lattice paths that connect the points in $S$ to those in $E$, but it is certainly possible that two different tilings give rise to two families of paths in which the connectivity of the start and end points differs. For each $\sigma\in\mathfrak{S}_{|S|}$ let $E_{\sigma}:=(e_{\sigma(1)},e_{\sigma(2)},\dots,e_{\sigma(|E|)})$ and suppose $N(S,E_{\sigma})$ denotes the total number of families of non-intersecting lattice paths in which each point $s_i\in S$ is joined to $e_{\sigma(i)}\in E_{\sigma}$ (it may well be that $N(S,E_{\sigma})$ is zero for certain $\sigma$). 

Lindstr\"om~\cite{Lindstrom73} (and later, Gessel and Viennot~\cite{GesselViennot89}) showed that
\begin{equation}\label{eq:LGV}
\sum_{\sigma\in\mathfrak{S}_{|S|}}\sgn(\sigma)N(S,E_{\sigma})=\pm\det(P_{S,E}),
\end{equation}
where $P_{S,E}=(P_{i,j})_{1\le i,j\le |S|}$ is the \emph{lattice path matrix} corresponding to $S$ and $E$ with entries given by the number of non-intersecting lattice paths that begin at $s_i$ and end at $e_j$.

\begin{rmk}\label{rmk:Binom}The number of lattice paths beginning at $(x_1,y_1)$ and ending at $(x_2,y_2)$ is given by the binomial coefficient
$$\binom{x_2-x_1+y_2-y_1}{x_2-x_1},$$
where
$$\binom{n}{k}:=\begin{cases}\frac{n!}{(n-k)!k!}&n,k\in\mathbb{N}_0, k\leq n,\\ 0 &otherwise\end{cases}$$
(in the above $\mathbb{N}_0:=\mathbb{N}\cup\{0\}$).	At first sight this may seem like a somewhat unnatural definition of the binomial coefficient, however a moment's thought convinces us that this definition is in fact completely natural within this context. We are enumerating lattice paths consisting of unit steps in the north and east directions on $\mathbb{Z}_{a,c}\times\mathbb{Z}_{a,b}$ and thus it \emph{should} give zero for points that are not separated by unit steps, and also for those pairs of points for which the end point is located to the left of, or below, the start point. We therefore interpret all binomial coefficients within this article in the same way.
\end{rmk}

\section{Combining the two approaches}\label{sec:Comb}

We return now to our expression for the $(i,j)$-entry of the inverse Kasteleyn matrix $A_{G}^{-1}$ from the end of Section~\ref{sec:PerfMatch},
$$(-1)^{i+j}\cdot\frac{\det(A_{G\setminus\{b_j,w_i\}})}{\det(A_G)}.$$
The graph $G\setminus\{b_j,w_i\}\subset\mathscr{H}$ corresponds to a hexagon $H\setminus\{\triangleleft,\triangleright\}\subset\mathscr{T}$, where $\triangleleft$ is the left pointing unit triangle corresponding to $b_j$ and $\triangleright$ the right pointing triangle corresponding to $w_i$.

If we denote by $(r_x,r_y)$ and $(l_x,l_y)$ the start and end points generated by the removal of $w_i$ and $b_j$ respectively\footnote{It can easily be checked that these coordinates are given by $\psi(\triangleright)$ and $\psi(\triangleleft)$.} then according to Section~\ref{sec:RhombTil} each perfect matching of $G\setminus\{b_j,w_i\}$ corresponds to a certain family of non-intersecting lattice paths that begin at the set of points $S^{w_i}:=\{(i-\tfrac{1+a+c}{2},\tfrac{a-b+1}{2}-i):1\le i\le a\}\cup\{(r_x,r_y)$ and end at the set of points $E^{b_j}:=\{(j-\tfrac{1+a-c}{2},\tfrac{a+b+1}{2}-j):1\le j\le a\}\cup\{(l_x,l_y)\}$.  Suppose we order our start points $S^{w_i}:=(s_1,s_2,\dots,s_{a+1})$ so that
$$s_i:=\begin{cases}
(i-\tfrac{1+a+c}{2},\tfrac{a-b+1}{2}-i) & 1\le i\le a,\\
(r_x,r_y) & i=a+1,
\end{cases}$$
and at the same time order our end points $E^{b_j}:=(e_1,e_2,\dots,e_{a+1})$, so that
$$e_j:=\begin{cases}
(j-\tfrac{1+a-c}{2},\tfrac{a+b+1}{2}-j) & 1\le j\le a,\\
(l_x,l_y) & j=a+1.
\end{cases}$$

We may then construct the lattice path matrix $P_{S,E}^{w_i,b_j}=(P_{i,j})_{1\le i,j\le a+1}$ to be the $(a+1)\times(a+1)$ matrix with $(i,j)$-entry given by the number of paths from $s_i\in S^{w_i}$ to $e_j\in E^{b_j}$,  that is,
$$P_{i,j}:=\begin{cases}
				\binom{b+c}{c+j-i}&1\le i,j\le a,\\
				\binom{(b+c)/2-r_x-r_y}{j-r_x-(a-c+1)/2}& i=a+1,1\le j\le a\\
				\binom{l_x+l_y+(b+c)/2}{l_x-i+(a+c+1)/2} & 1\le i\le a, j=a+1,\\
				\binom{l_x+l_y-r_x-r_y}{l_x-r_x} & i=j=a+1,	
			\end{cases}$$
It follows from~\cite{Lindstrom73} and~\cite{GesselViennot89} that $\det(P_{S,E}^{w_i,b_j})$ gives (up to sign) the following sum over non-intersecting paths
\begin{equation}\label{eq:SgnSumP}
\sum_{\sigma\in\mathfrak{S}_{a+1}}\sgn(\sigma) N(S^{w_i},E^{b_j}_{\sigma}),
\end{equation}
where $N(S^{w_i},E^{b_j}_{\sigma})$ is the number of non-intersecting lattice paths that begin at $S^{w_i}$ and end at $E_{\sigma}^{b_j}:=(e_{\sigma(1)},e_{\sigma(2)},\dots,e_{\sigma(a+1)})$.

If we look a little closer we see that what we have in this expression is a sum over different families of non-intersecting lattice paths, some of which contribute negatively and some of which contribute positively. If $P^+$ denotes the set of all families that make a positive contribution while $P^-$ denotes those families that make a negative one then we have
$$\pm\det(P_{S,E}^{w_i,b_j})=|P^+|-|P^-|.$$

Consider now $\det(A_{G\setminus\{b_j,w_i\}})$, which may be written as 
\begin{equation}\label{eq:SgnSumM}
\sum_{\pi\in\mathfrak{S}_{ab+bc+ca-1}}\sgn(\pi)P_m(B',W_{\pi}'),
\end{equation}
in which $P_m(B',W_{\pi}')$ denotes the number of perfect matchings where the $i$-th vertex in $B'$ (here $B':= (b_1,b_2,\dots,b_{ab+bc+ca-1})$ is the tuple of labelled black vertices of $G\setminus\{b_j,w_i\}$) is matched with the $i$-th vertex in $W_{\pi}'$ ($W_{\pi}':=(w_{\pi(1)},w_{\pi(2)},\dots,w_{\pi(ab+bc+ca-1)})$ is the tuple of labelled white vertices of $G\setminus\{b_j,w_i\}$). 

As with our expression for $\det(P_{S,E}^{w_i,b_j})$ above we may write~\eqref{eq:SgnSumM} as a sum over sets of perfect matchings, some of which contribute positively and some of which contribute negatively to the sum. By letting $M^+$ denote the set of all matchings that make a positive contribution and $M^-$ the set of all those that make a negative one, we see that
$$\pm\det(A_{G\setminus\{b_j,w_i\}})=|M^+|-|M^-|.$$

How, then, may we relate the sets of families of lattice paths $P^+$ and $P^-$ to the set of perfect matchings $M^-$ and $M^+$? We already know that the union $P^+\cup P^-$ is in bijection with $M^+\cup M^-$, however in 2015 Cook and Nagel~\cite{CookNagel15} refined this bijection even further, successfully showing that the families of paths in $P^+$ are in bijection with either those matchings in $M^+$, or instead with those in $M^-$, thus
\begin{equation}
\label{eq:AeqP}\det(A_{G\setminus\{b_j,w_i\}})=\pm\det(P_{S,E}^{w_i,b_j})\end{equation} and our goal now is to find a closed expression for $\det(P_{S,E}^{w_i,b_j})$.
\begin{rmk}
	Cook and Nagel in fact refined the bijection between signed lattice paths and signed perfect matchings for more general triangular regions of the triangular lattice, however it is easy to see that the hexagonal region $H$ may be obtained by cutting off corners from a larger triangular region.
\end{rmk}
\begin{rmk}\label{rmk:COnventionM}
	It should be noted that the sign in~\eqref{eq:AeqP} can be controlled by labelling the vertices of $G\setminus\{b_j,w_i\}$ in a consistent manner. Consider the tuples of vertices from $G$, $B:=(b_1,b_2,\dots,b_{ab+bc+ca})$ and $W:=(w_1,w_2,\dots,w_{ab+bc+ca})$. Let us now remove a vertex $b_j$ from $B$ and $w_i$ from $W$ and let $B':=B\setminus\{b_j\},W':=W\setminus\{w_i\}$. We then re-label each element $b'_k\in B'$ according to the following convention
	$$b'_k:=\begin{cases}b_k & 1\le k< j,\\
	b_{k-1} & k > j,\end{cases}
	$$
	and similarly for those vertices in $W'$ . Since we have also fixed the labelling of our start and end points that index our lattice path matrix $P_{S,E}^{w_i,b_j}$, it follows that either 
	$$\det(A_{G\setminus\{b_j,w_i\}})=-\det(P_{S,E}^{w_i,b_j})$$
	for \emph{every} pair of vertices $\{b_j,w_i:b_j\in B, w_i\in W\}$, otherwise
	$$\det(A_{G\setminus\{b_j,w_i\}})=\det(P_{S,E}^{w_i,b_j})$$
	for all such pairs.
\end{rmk}
\section{An exact formula}\label{sec:Exact}
We shall now derive a closed form expression for $\det(P_{S,E}^{w_i,b_j})$ by finding the $LU$-decomposition of our lattice path matrix. The following result was guessed using the computer software package {\tt Rate}\footnote{This Mathematica package was created by C.~Krattenthaler and is available at~\url{http://www.mat.univie.ac.at/~kratt/rate/rate.html}.} (``Guess" in German) and its proof relies partly on a computer implementation of Zeilberger's algorithm\footnote{A Mathematica implementation of this algorithm is available at~\url{http://www.risc.jku.at/research/combinat/software}.} (see~\cite{PauleSchorn95}).
\begin{prop}
	The lattice path matrix $P_{S,E}^{w_i,b_j}$ has $LU$-decomposition
	$$P_{S,E}^{w_i,b_j}=L\cdot U$$
	where $L=(L_{i,j})_{1\le i,j\le a+1}$ has entries given by
	$$L_{i,j}:=\begin{cases}A(b,c,i,j)& 1\le i,j\le a,\\
	 B(a,b,c,r_x,r_y,j)& i=a+1,1\le j\le a,\\
	1& otherwise,
	\end{cases}$$
	and $U=(U_{i,j})_{1\le i,j\le a+1}$ is given by
	$$U_{i,j}:=\begin{cases}
	C(b,c,i,j)& 1\le i,j\le a,\\
	D(a,b,c,l_x,l_y,i) & 1\le i\le a, j=a+1,\\
	\binom{l_x+l_y-r_x-r_y}{l_x-r_x}-\sum_{v=1}^aB(v)\cdot D(v)&i=j=a+1,\\
	0& \text{otherwise,}
	\end{cases}$$
	with
	\begin{align*}
	A(b,c,i,j):=&\frac{c! (i-1)! (b+j-1)!}{(j-1)! (b+i-1)! (i-j)! (c-i+j)!},\\
	B(a,b,c,r_x,r_y,j):=&\sum_{v=1}^{j}\frac{(-1)^{j-v} (b+j-1)! (c+v-1)! (b+j-v-1)!}{(b-1)! (v-1)! (j-v)! (b+c+j-1)!}\binom{\frac{b}{2}+\frac{c}{2}-r_x-r_y}{v-\frac{1}{2} (a-c+1)-r_x},\\
	C(b,c,i,j):=&\frac{b! (j-1)! (b+c+i-1)!}{(b+i-1)! (c+j-1)! (j-i)! (b+i-j)!},\\
	D(a,b,c,l_x,l_y,i):=&\sum_{v=1}^{i}\frac{(i-1)! (-1)^{i-v} (b+v-1)! (c+i-v-1)!}{(c-1)! (v-1)! (b+i-1)! (i-v)!}\binom{\frac{b}{2}+\frac{c}{2}+l_x+l_y}{\frac{1}{2} (a+c+1)+l_x-v}.
	\end{align*}	
\end{prop}
\begin{proof}
	It is easy to see that $$\sum_{s=1}^{a+1}L_{a+1,s}\cdot U_{s,a+1}=\binom{l_x+l_y-r_x-r_y}{l_x-r_x},$$
	thus in order to complete the proof we must show the following:
	\begin{enumerate}[(i)]
		\item $\sum_{s=1}^{\min\{i,j\}}A(b,c,i,s)\cdot C(b,c,s,j)=\binom{b+c}{c+j-i}$;
		\item $\sum_{s=1}^{i}A(b,c,i,s)\cdot D(a,b,c,l_x,l_y,s)=\binom{l_x+l_y+(b+c)/2}{l_x-i+(a+c+1)/2}$;
		\item $\sum_{s=1}^{j}B(a,b,c,r_x,r_y,s)\cdot C(b,c,s,j)=\binom{(b+c)/2-r_x-r_y}{j-r_x-(a-c+1)/2}$.
	\end{enumerate}
	
	In the first case we have
	\begin{multline*}(b+i-j+1)\sum_{s=1}^{i+1}A(b,c,i+1,s)\cdot C(b,c,s,j)\\+(i-j-c)\sum_{s=1}^{i}A(b,c,i,s)\cdot C(b,c,s,j)=0,\end{multline*}
	and
	\begin{multline*}(c-i+j+1)\sum_{s=1}^{j+1}A(b,c,i,s)\cdot C(b,c,s,j+1)\\+(j-i-b)\sum_{s=1}^{j}A(b,c,i,s)\cdot C(b,c,s,j)=0.\end{multline*}
	It is straightforward to check that
	$$(b+i-j+1)\binom{b+c}{c+j-i-1}+(i-j-c)\binom{b+c}{b+j-i}=0$$
	and
	$$(c-i+j+1)\binom{b+c}{c+j-i+1}+(j-i-b)\binom{b+c}{c+j-i}=0,$$
	so the identity holds once the initial conditions for the recurrence have been verified.
	
	For the second case note that by interchanging the summations we obtain
	$$\sum_{v=0}^{i-1}\sum_{s=v}^{i-1}\frac{c(-1)^{s-v} (i-1)! (b+v)! (c+s-v-1)!}{ v! (b+i-1)! (i-s-1)! (s-v)! (c-i+s+1)! }\binom{\frac{b}{2}+\frac{c}{2}+l_x+l_y}{\frac{a}{2}+\frac{c}{2}+\frac{1}{2}+l_x-v}$$
	the inner sum of which may in turn be expressed as a $_2F_1$ hypergeometric series\footnote{The $_pF_q$ hypergeometric series, denoted $\displaystyle\pFq{p}{q}{a_1,\dots,a_p}{b_1,\dots,b_q}{z}$, is defined to be $\displaystyle\sum_{k=0}^{\infty}\frac{(a_1)_k\cdots(a_p)_k}{(b_1)_k\cdots(b_q)_k}\frac{z^k}{k!}$, where $(\alpha)_{\beta}$ is the Pochhammer symbol, that is, $(\alpha)_{\beta}:=\alpha\cdot(\alpha+1)\cdots(\alpha+\beta-1)$ for $\beta> 0$, while $(\alpha)_0:=1$. \label{fn:Poch}}
	$$\sum_{v=0}^{i-1}\frac{c!(i-1)!(b+v)!}{v!(b+i-1)!(i-v-1)!(c-i+v+1)!}\binom{\frac{b}{2}+\frac{c}{2}+l_x+l_y}{\frac{a}{2}+\frac{c}{2}+\frac{1}{2}+l_x-v}\pFq{2}{1}{c,v-i+1}{c-i+v+2}{1}.$$
	
	When faced with an expression such as this there is a dearth of transformation and summation identities that one may turn to in order to try to simplify things. In the expression above it turns out that a straightforward application of the Chu-Vandermonde identity, 
	$$\pFq{2}{1}{a,-n}{c}{1}=\frac{(c-a)_n}{(c)_n},$$
	(which may be found in~\cite[1.7.7; Appendix III.4]{Slat66}) yields
	$$\frac{(v-i+2)_{i-v-1}}{(c-i+v+2)_{i-v-1}},$$
	where $(\alpha)_{\beta}$ is the Pochhammer symbol (see footnote~\ref{fn:Poch}). For $v<i-1$ the above term vanishes, thus proving (ii). 
	
	Precisely the same approach can be used to prove the third case (that is, interchanging the sums and applying the Chu-Vandermonde identity), thus it suffices to say that once this last identity has been verified the proof is complete.
\end{proof}
This proposition immediately gives rise to the following Corollary.

\begin{cor}
	The determinant of the lattice path matrix $P_{S,E}^{w_i,b_j}$ is given by
	$$M(H)\cdot\left(\binom{l_x+l_y-r_x-r_y}{l_x-r_x}-\sum_{v=1}^aB(v)\cdot D(v)\right).$$
\end{cor}
This follows from the fact that $L_{i,i}=1$ for $1\le i\le a+1$, thus the determinant of $P_{S,E}^{w_i,b_j}$ is the product of the diagonal entries of $U$,
$$\left(\prod_{i=1}^{a}A(b,c,i,i)\cdot C(b,c,i,i)\right)\cdot U_{a+1,a+1}.$$
The product on the left of this expression may be re-written as
$$\prod_{i=1}^a\prod_{j=1}^b\prod_{k=1}^c\frac{i+j+k-1}{i+j+k-2}$$which we instantly recognise as MacMahon's formula~\cite{MacMahon16} that counts the number of tilings of the hexagon $H$ (see Section~\ref{sec:PerfMatch}).

Everything is now in place for us to state the main result of this article, which follows from inserting our expression for $\det(P_{S,E}^{w_i,b_j})$ into our expression for the entries of $A_G^{-1}$ from Section~\ref{sec:PerfMatch}.

\begin{thm}\label{thm:Main}
	The inverse Kasteleyn matrix corresponding to the sub-graph $G$ of the hexagonal lattice $\mathscr{H}$ consisting of black and white vertices ($\{b_1,b_2,\dots,b_{ab+bc+ca}\}$ and $\{w_1,w_2,\dots,w_{ab+bc+ca}\}$ respectively) is equal to $(\pm1)\cdot K$, where $K=(K_{w_i,b_j})_{w_i,b_j\in G}$ is the matrix with entries given by
	$$(-1)^{i+j}\cdot\left( \binom{j_x+j_y-i_x-i_y}{j_x-i_x}-\sum_{t=1}^{a}\displaystyle\frac{g(a,b,c,j_x,j_y,t)g(a,c,b,-i_y,-i_x,t)}{\binom{b+c+t-1}{b+t-1}\binom{b+t-1}{t-1}}\right),$$
	in which
	$$g(u,v,w,x,y,z):=\displaystyle\sum_{s=1}^z(-1)^{z-s}\binom{v+s-1}{s-1}\binom{w+z-s-1}{w-1}\binom{\frac{v}{2}+\frac{w}{2}+x+y}{x-s+\frac{u}{2}+\frac{w}{2}+\frac{1}{2}},$$
	and the points $(i_x,i_y),(j_x,j_y)\in\mathbb{Z}_{a,c}\times\mathbb{Z}_{a,b}$ are determined by the distance of the vertices $w_i$ and $b_j$ (respectively) from the centre of $G$.
\end{thm}
\begin{rmk}
	The vertices $(i_x,i_y)$ and $(j_x,j_y)$ in the above theorem are obtained by applying the function $\psi$ from Section~\ref{sec:RhombTil} to the triples that describe the unit triangles corresponding to $w_i,b_j$, according to the labelling outlined in Section~\ref{sec:RhombTil}.
\end{rmk}
\section{Applications of the main result}\label{sec:App}

Theorem~\ref{thm:Main} has a number of useful applications as it allows us to compute the number of tilings of $H\setminus T$ as the determinant of a matrix whose size is dependent on $T$. By considering particular families of holes not only can we recover existing results, but at the same time we are afforded an entirely new position from which we may attack various problems that lie at the boundary of combinatorics and statistical physics.

\subsection{Exact enumeration of tilings}
Suppose $\{b_j,w_i\}$ is an admissibility inducing set of vertices contained in $G$. By translating $G\setminus \{b_j,w_i\}$ into an hexagonal region on $\mathscr{T}$ we see that $\{b_j,w_i\}$ correspond to either a pair of unit triangles that share an edge (and so form a rhombus) or meet at a point (forming a unit triangular ``bow tie"). Otherwise $b_i,w_j$ are a pair of vertices induce a larger set of holes that have even charge. 

According to Lemma~\ref{lem:AdmissPres} the number of such tilings is given by $$\prod_{i=1}^a\prod_{j=1}^b\prod_{k=1}^c\frac{i+j+k-1}{i+j+k-2}\cdot|K_{w_i,b_j}|.$$ 

\begin{rmk}\label{rmk:Fisch}If we specify $b_j,w_i$ so that the corresponding triangles form a horizontal rhombus in $H\setminus T$ then we recover a result of Fischer found in~\cite{Fischer01}, whereas for $b_j$ and $w_i$ forming a bow tie we obtain a generalisation of Eisenk\"olbl's result~\cite{Eisen99a}.\end{rmk}

This idea may be extended by way of Kenyon's result~\cite{Kenyon97} (see Section~\ref{sec:PerfMatch}) so that if $V:=\{b_1,b_2,\dots,b_k,w_1,w_2,\dots,w_k\}$ corresponds to a specific arrangement of rhombi or bow ties in $H\setminus T$ then
$$M(H\setminus T)=\prod_{i=1}^a\prod_{j=1}^b\prod_{k=1}^c\frac{i+j+k-1}{i+j+k-2}\cdot |\det(K_V)|,$$
where $K_V=(K_{w_i,b_j})_{w_i,b_j\in V}$ is the sub-matrix obtained by restricting $K$ to those rows and columns indexed by the vertices in $V$.

If $V$ instead corresponds to a set of unit triangles that lie along the outer boundary of $H\setminus T$ then the above expression gives exactly Ciucu's generalisation of Kuo condensation~\cite{Ciucu15} for rhombus tilings of $H$. Further to this, $V$ could correspond to holes contained in the interior of $T$, in which case we can recover equivalent expressions to those found in articles by the author~\cite{Gilmore15,Gilmore16}, Ciucu and Fischer~\cite{CiucuFischer16}. 

If we select a set of admissibility inducing vertices $V$ in the correct way (so that certain regions of the interior of $H\setminus T$ are forced) then we also have an alternative method for deriving results from Ciucu and Krattenthaler~\cite{CiucuKratt13} and Eisenk\"olbl (together with others)~\cite{Eisen01}, in which the authors enumerate tilings of hexagons that are not semi-regular (also known as \emph{unbalanced}) and contain holes in their interior (see Figure~\ref{fig:Central}, left). This also answers the open problem posed in~\cite{CiucuFischer16}, since we may remove from $H\setminus T$ any set of even charge inducing holes, and this includes unit triangles that lie along its outer boundary (see Figure~\ref{fig:Central}, centre).

\begin{figure}
	\includegraphics[width=0.3 \textwidth]{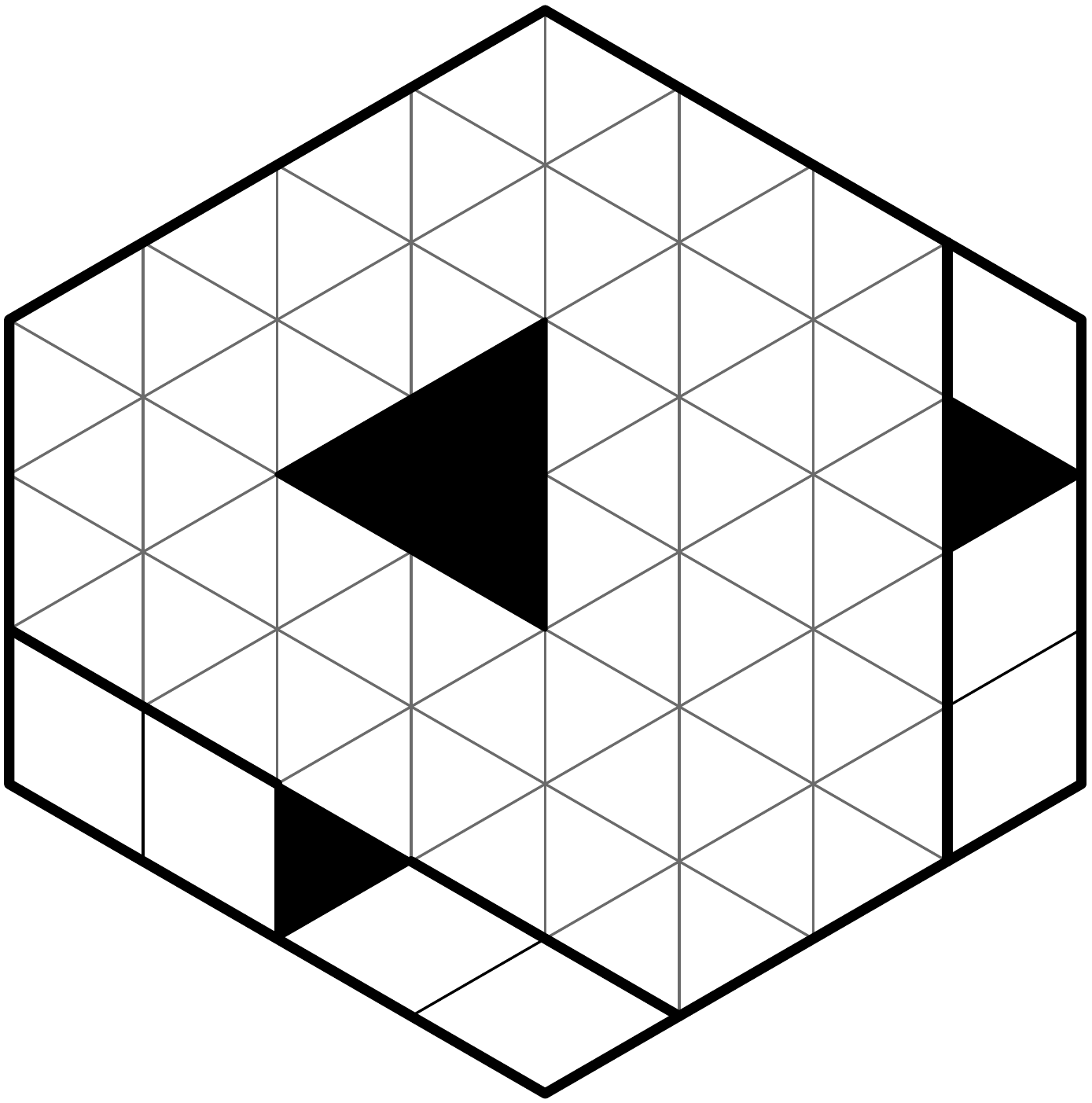}\,\,\,\,\,\,\,\,\,\,\includegraphics[width=0.3 \textwidth]{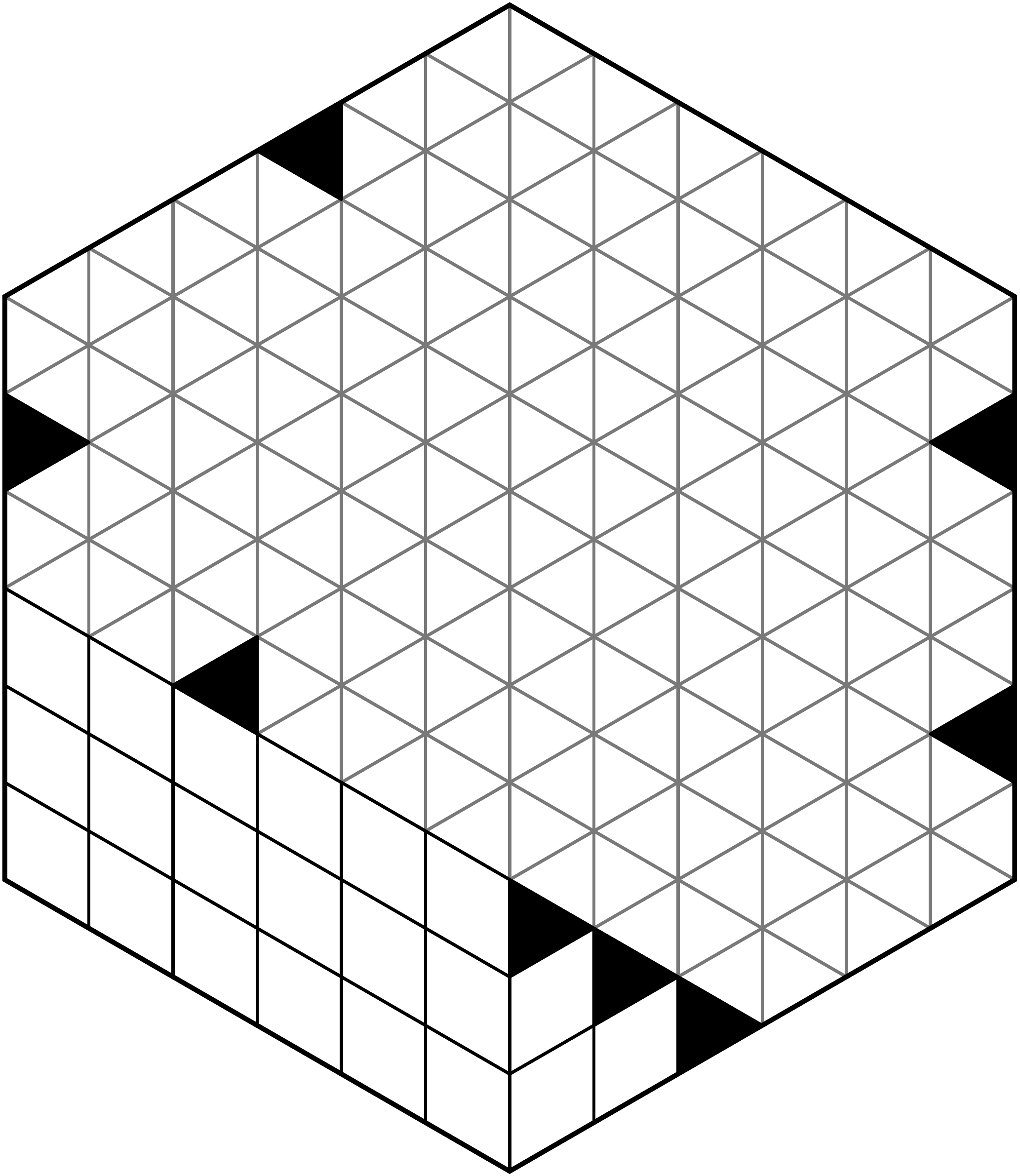}\,\,\,\,\,\,\,\,\,\,\includegraphics[width=0.3 \textwidth]{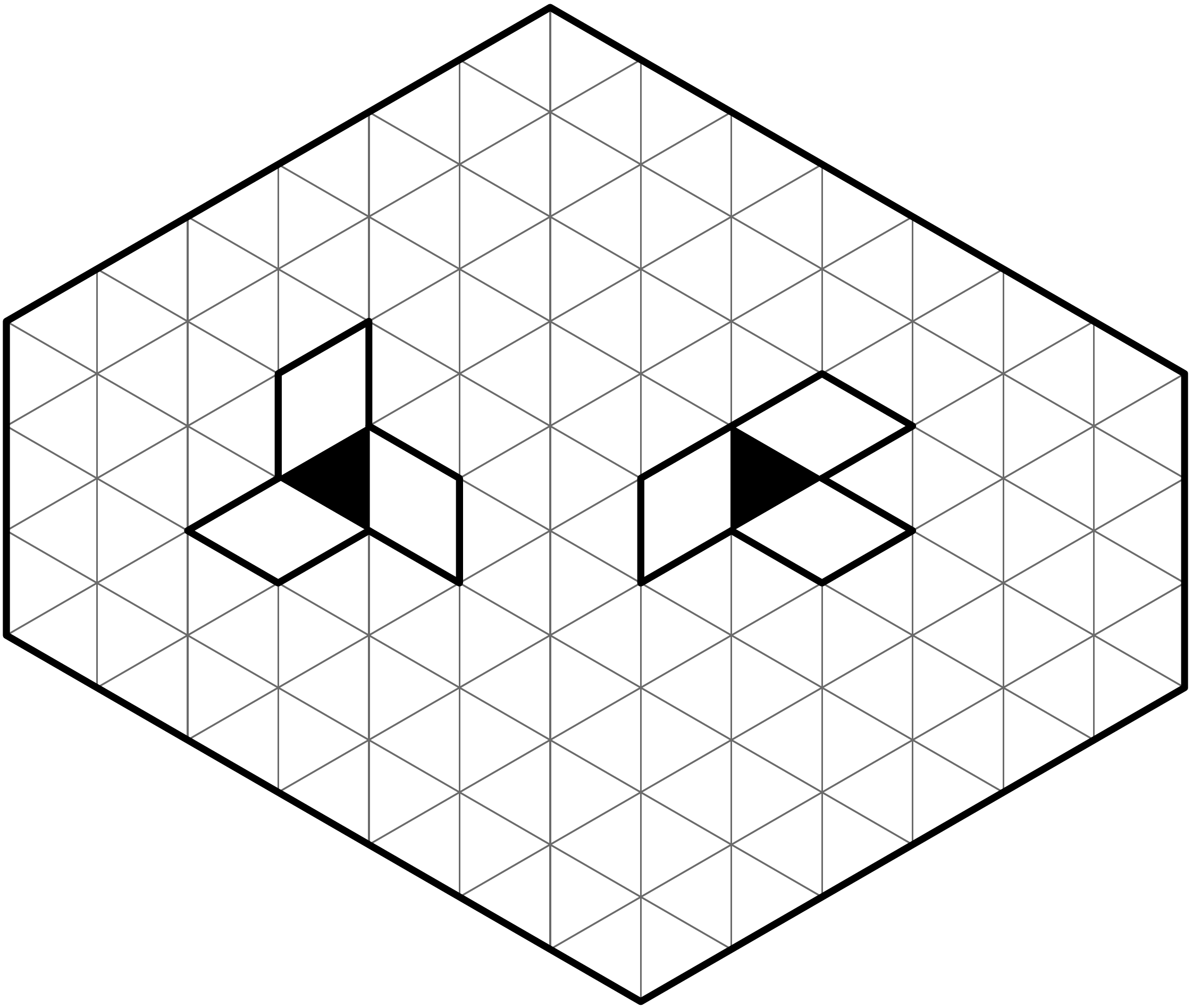}\caption{A hexagon of side lengths $5,2,4,3,4,2$ with a central hole obtained by removing triangles from the boundary that force tiles along the edges (left), an unbalanced hexagon with arbitrary dents (centre), and one arrangement of six rhombi that may be found in a rhombus tiling of $H_{7,3,6}\setminus T$ where $T$ is a pair of unit triangles (right).}\label{fig:Central}
\end{figure}

Theorem~\ref{thm:Main} therefore unites a large number of existing enumerative formulas for different classes of holes under one roof. Of course, unless we remove one pair of even charge inducing vertices then we must still compute a determinant, however the fact remains that the size of this determinant is constant for a fixed set of holes, irrespective of how much we vary the size of the region in which they are contained.

\subsection{Statistics of rhombi and correlations of holes}

Theorem~\ref{thm:Main} also has a number of potential applications that are of a statistical physics flavour, although this relies on obtaining the asymptotics of the entries of $K$ as the boundary of $H$ is sent to infinity. Successfully extracting these asymptotics would in the first instance give the probability of each unit rhombi occurring in a random tiling, thus yielding an analogous result to that of Kenyon~\cite{Kenyon97} for planar tilings (as opposed to those embedded on the torus). 

In turn this may lead to proofs of certain conjectures about the correlation of holes in a ``sea of unit rhombi". The \emph{correlation} of a set of holes $T$ is given by
$$\omega(T):=\lim_{n\to\infty}\frac{M(H\setminus T)}{M(H)}.$$
In 2008 Ciucu~\cite{Ciucu08} conjectured that if the distance between the holes is proportional to some real $\kappa$ then as $\kappa\to\infty$,
$$\omega(T)=\prod_{t\in T}C_t\prod_{1\le i<j\le |T|}\text{d}(t_i,t_j)^{\frac{1}{2}q(t_i)q(t_j)},$$
where $\text{d}(t_i,t_j)$ is the Euclidean distance between the holes $t_i,t_j$, $q(t_i)$ is the charge of the hole $t_i$ and $C_t$ is a constant dependent on each hole $t$.

Provided the entries of $K$ in the limit (and as the distance of the holes grows large) is such that the determinant of $K_V$ has a straightforward evaluation (as was the case in~\cite{Gilmore16}), Theorem~\ref{thm:Main} could well lead to a proof of Ciucu's conjecture for the most general class of holes to date.

One further application could be an alternative proof of the hexagonal lattice analogy of the conjecture of Fisher and Stephenson~\cite{FishSteph63}. This is a special case of Ciucu's conjecture where $T$ consists of a pair of unit triangles $\triangleleft,\triangleright$ and was recently proved by Dub\'{e}dat~\cite{Dub15}, however Theorem~\ref{thm:Main} may also offer a different approach to the same problem.

Although our approach does not immediately allow us to enumerate the number of tilings of a region containing a pair of unit triangular holes (unless they are even charge inducing), we may express the number of tilings of $H\setminus\{\triangleleft,\triangleright\}$ as a sum
$$\sum_{T\subset H}|\det(K_{V_T})|,$$
where the sum is taken over subsets $T$ that correspond to the different arrangements of unit rhombi whose edges coincide with the edges of $\triangleleft,\triangleright$ ($V_T$ is the corresponding sub-matrix of $K$ for each arrangement, see Figure~\ref{fig:Central}, right, for an example of one such arrangement). This would yield a sum consisting of $2^6$ terms,each involving a $6\times 6$ determinant evaluation, so supposing again that each determinant has a straightforward evaluation we would therefore obtain completely different proof to that found in~\cite{Dub15}.

\printbibliography

\end{document}